\documentclass[11pt]{amsart}

\usepackage[latin1]{inputenc}
\usepackage{csquotes}
\usepackage[english]{babel}
\usepackage[T1]{fontenc}

\usepackage{indentfirst}
\usepackage{enumitem}
\usepackage{amsmath,amssymb, amsbsy}
\usepackage{comment}
\usepackage{amsfonts}
\usepackage{esint}
\usepackage{latexsym}
\usepackage{amsthm}
\usepackage[dvips]{graphicx}
\usepackage{xcolor}

\usepackage[outline]{contour}
\DeclareGraphicsExtensions{.pdf,.png,.jpg,.eps}
\usepackage{stackengine,scalerel}

\usepackage{multicol}
\usepackage{tikz-cd}

\usepackage{bbm}

\usepackage[font=small,labelfont=bf]{caption}

\usepackage{amsthm}

\usepackage{lipsum}

\usepackage[active]{srcltx}
\definecolor{Arancio}{cmyk}{0,0.61,0.87,0}

\definecolor{blus}{RGB}{0,102,204}

\usepackage{doi}

\newcommand{\brd}[1]{\mathbb{#1}}
\newcommand{\R}{\brd{R}}

\newcommand{\N}{\brd{N}}
\newcommand{\LM}{\mathbb{L}}

\newcommand{\h}{\brd{H}}
\newcommand{\B}{\brd{B}}
\newcommand{\LL}{\mathbb{L}}
\newcommand{\s}{\brd{S}}

\newcommand{\bnabla}{\overline{\nabla}}
\newcommand{\second}{\mathrm{II}}
\newcommand{\X}{\mathfrak{X}}

\newcommand{\tr}{\mathrm{tr}}
\newcommand{\bpart}{\overline{\partial}}
\newcommand{\lie}{\mathcal{L}}

\newcommand{\grad}{\nabla_{\!\delta}}
\newcommand{\hess}{\nabla^2_{\!\delta\,}}

\newcommand{\abs}[1]{\left\lvert {#1} \right\rvert}
\newcommand{\norm}[2]{\left\Vert {#1} \right\Vert_{#2}}

\newcommand{\be}{\begin{equation}}
\newcommand{\ee}{\end{equation}}

\numberwithin{equation}{section}

\newtheorem{teo}{Theorem}[section]

\newtheorem{Lemma}[teo]{Lemma}
\newtheorem{Theorem}[teo]{Theorem}
\newtheorem{Proposition}[teo]{Proposition}
\theoremstyle{definition}
\newtheorem{Definition}[teo]{Definition}

\renewcommand{\div}{\operatorname{div}}

\usepackage{tikz}
\usepackage{mathdots}
\usepackage{yhmath}
\usepackage{cancel}
\usepackage{siunitx}
\usepackage{array}
\usepackage{multirow}
\usepackage{textcomp}
\usepackage{gensymb}
\usepackage{tabularx}
\usepackage{extarrows}
\usepackage{booktabs}
\usetikzlibrary{fadings}
\usetikzlibrary{patterns}
\usetikzlibrary{shadows.blur}
\usetikzlibrary{shapes}

\usepackage{hyperref}

\title[Entire spacelike radial graphs with prescribed mean curvature in $\LM^{m+1}$]{Entire spacelike radial graphs with prescribed mean curvature in the Lorentz--Minkowski space}

\begin{document}

\author{G\MakeLowercase{abriele} Cora, A\MakeLowercase{lessandro} Iacopetti, L\MakeLowercase{orenzo} Maniscalco}

\thanks{\emph{Acknowledgements.} Research partially supported by Gruppo Nazionale per l'Analisi Matematica, la Pro\-ba\-bi\-li\-t\`a e le loro Applicazioni (GNAMPA) of the Istituto Nazionale di Alta Matematica (INdAM). Alessandro Iacopetti is partially supported by the PRIN 2022 PNRR project 2022R537CS \emph{$NO^3$ - Nodal Optimization, NOnlinear elliptic equations, NOnlocal geometric problems, with a focus on regularity},  founded by the European Union - Next Generation EU and by the GNAMPA 2025 project E53C25002010001: ``Strutture Analitiche e Geometriche in PDEs: Regolarit\`a, Fenomeni Critici e Dinamiche Complesse".}

\address[Gabriele Cora]{D\'epartement de math\'ematique, Universit\'e Libre de Bruxelles, Campus de la Plaine - CP214 boulevard du Triomphe, 1050 Bruxelles, Belgium}
\email{gabriele.cora@ulb.be }

\address[Alessandro Iacopetti]{Dipartimento di Matematica ``Giuseppe Peano", Universit\`a degli Studi di Torino, Via Carlo Alberto 10, 10123 Torino, Italy}
\email{alessandro.iacopetti@unito.it}

\address[Lorenzo Maniscalco]{Dipartimento di Matematica ``Giuseppe Peano", Universit\`a degli Studi di Torino, Via Carlo Alberto 10, 10123 Torino, Italy}
\email{lorenzo.maniscalco@unito.it}

\begin{abstract}
In this paper we address the existence and uniqueness of entire spacelike hypersurfaces in the Lorentz--Minkowski space $\LM^{m+1}$ with prescribed mean curvature that are star-shaped with respect to a point and asymptotic to a light cone. We also establish a Willmore-type inequality and prove a non-existence result for spacelike radial graphs asymptotic to the light cone whose mean curvature belongs to $L^p$ for $1 \leq p\leq m$, in particular in the case of compactly supported mean curvature.
\end{abstract}

\maketitle

\section{Introduction}

\subsection{The problem} 
Let $m\in\mathbb{N}$ with $m\geq 2$, and fix a point $o \in \mathbb{R}^{m+1}$. A radial graph centered at $o$ is a hypersurface $\Sigma \subset \mathbb{R}^{m+1}$ such that each ray emanating from $o$ intersects $\Sigma$ at most once. 

In the Euclidean setting, the problem of finding radial graphs with prescribed mean curvature, as well as the study of their qualitative properties, has been extensively investigated for disc-type surfaces (see, e.g., \cite{Rado, Lopez, HerbDel, CaldIac, Tausch, Serrin}, and the references therein).

For closed hypersurfaces (i.e., compact and without boundary), Treibergs and Wei proved in \cite{TreibergsWei1983} an existence and uniqueness result for radial graphs. 

To state their result, let $P$ denote the normalized position vector field from an origin $o\in\R^{m+1}$, $\rho$ the distance function from $o$, and $\s^m(r)$ the sphere of radius $r$ centred at the zero vector; for simplicity, when $r = 1$, we write $\s^m$ instead of $\s^m(1)$. The result reads as follows.

\begin{Theorem}[{\cite[Theorem~(a), (b)]{TreibergsWei1983}}]\label{thm:treibergs_wei}
    Suppose $\bar H\in C^1(\R^{m+1}\setminus \{o\})$ is a positive function that satisfies
    \begin{align}\label{eq:hp_treibergs_wei}
        P(\rho \bar H(q\rho)) \le 0\,, \qquad \text{ for all } q\in\s^m\,.
    \end{align}
    Assume also that there exist two radii $0 < r \le 1 \leq R$ such that 
    \begin{align*}
        \bar H|_{\s^m(r)} > \frac{1}{r}, 
        \qquad 
        \bar H|_{\s^m(R)} < \frac{1}{R}.
    \end{align*}
    Then there exists $u\in C^{2,\alpha}(\s^m)$ for some $\alpha\in(0,1)$ such that the mean curvature $H_u$ of the associated \emph{radial graph} centred in $o$
    \[
        \Sigma_u = \{ o + q e^{u(q)} \; ;\ q\in\s^m\}
    \]
    satisfies
    \begin{align}\label{eq:prescription_treibergs_wei}
        H_u = \bar H|_{\Sigma_u}.
    \end{align}
    Moreover, if \eqref{eq:hp_treibergs_wei} holds and there are two solutions $u,v\in C^2(\s^m)$ of \eqref{eq:prescription_treibergs_wei}, then they differ by a constant; that is, the corresponding radial graphs differ by a dilation.
\end{Theorem}
Throughout this paper we adopt the convention that the mean curvature is the arithmetic mean of the principal curvatures, with the sign convention chosen so that, for instance, the mean curvature of $\s^m(r)$ is $r^{-1}$. Equation \eqref{eq:prescription_treibergs_wei} is then a quasilinear elliptic equation for functions on the unit sphere which, in terms of spherical geometry, can be written as (see \cite{Lopez, TreibergsWei1983})
\[
-\div_{\sigma} \left(\frac{D u}{\sqrt{1+|D u|^2}}\right)
+ \frac{m}{\sqrt{1+|D u|^2}}
= m e^u \bar H(e^u q) \qquad \text{on } \s^m,
\]
where the divergence $\div_\sigma$ and the gradient $D$ are taken with respect to the standard metric $\sigma$ on $\s^m$.

It is worth noting that not every function $\bar H$ gives rise to a solution of \eqref{eq:prescription_treibergs_wei}, for the mean curvature $H$ of a \emph{closed} hypersurface $\Sigma$ must obey some geometrical constraint, one of which is the  Willmore inequality
\[
\int_{\Sigma} |H|^m \, dv \ge m\, |\mathbb{B}^m|,
\]
where $\mathbb{B}^m$ denotes the unit ball of $\mathbb{R}^m$ and $dv$ is the Riemannian volume form induced on $\Sigma$.

\bigskip

In this paper we consider spacelike radial graphs in the Lorentz--Minkowski space $\LM^{m+1}$, that is, $\mathbb{R}^{m+1}$ endowed with the scalar product
\begin{align*}
    \eta = \langle \cdot,\cdot \rangle = - (dx^0)^2 + (dx^1)^2 + \dots + (dx^m)^2.
\end{align*}
A vector $q\in \LM^{m+1}\setminus\{0\}$ is called \emph{spacelike}, \emph{timelike}, or \emph{lightlike} if $\langle q,q \rangle > 0$, $\langle q,q \rangle < 0$, or $\langle q,q\rangle = 0$, respectively. 
A hypersurface $\Sigma \subset \LM^{m+1}$ is said to be \emph{spacelike} if it is of class $C^1$ and the Minkowski metric $\eta$ induces a Riemannian metric on $\Sigma$; it is called \emph{entire} if it is the graph of a function $f:\mathbb{R}^m \to \mathbb{R}$. 

An example of an entire spacelike hypersurface is the future unit hyperboloid,
\begin{align*}
    \h^m := \{ q \in \LM^{m+1} \; ; \; \langle q,q\rangle = -1, \ \langle q, E_0 \rangle < 0 \},
\end{align*}
where $E_0 = (1,0,\dots,0)$.

It is well known that, if $h$ denotes the metric induced on the hyperboloid, then $(\h^m,h)$ is isometric to the hyperbolic space, hence the notation. 
We define the \emph{Lorentzian distance} (more properly, the \emph{time separation}) from a point $o\in\LM^{m+1}$ by
\begin{align}\label{eq:time_separation}
    \ell_o(q) := \sqrt{|\langle q-o,q-o\rangle|}, 
    \qquad q\in I^+(o),
\end{align}
where $I^+(o)$ denotes the chronological future of $o$ (see Section~\ref{sec:preliminaries}). 
When $o$ is the zero vector in $\LM^{m+1}$, we simply write $\ell$ in place of $\ell_0$. We then define the hyperboloids of radius $l>0$ as the level sets
\[
\h^m(l) := \{q\in I^+(0)\, ;\ \ell(q) = l \},
\]
so that $\h^m(1)=\h^m$. 
These hypersurfaces are the Lorentzian counterpart of the Euclidean spheres $\mathbb{S}^m(r)$. For instance, in accordance with our sign convention for the mean curvature, the hyperboloid $\h^m(l)$ has mean curvature equal to $l^{-1}$ (with respect to the future-pointing unit normal).
Motivated by this analogy, we call \emph{entire radial graphs} (centered at $o$) those hypersurfaces $\Sigma \subset \LM^{m+1}$ for which there exists a point $o\in \LL^{m+1}$ and a continuous function $u:\h^m\to\mathbb{R}$ such that
\[
\Sigma = \Sigma_u := \{ o + qe^{u(q)} \; ;\ q\in\h^m \}.
\]
In this case, $u$ is called the \emph{height function} of $\Sigma$. We say that $u$ is \emph{spacelike} if the associated hypersurface $\Sigma_u$ is spacelike.
For instance, the hyperboloid $\h^m(l)$ is the radial graph of the constant function $\ln l$ (and thus constant functions are spacelike).

While the notion of mean curvature for spacelike hypersurfaces is similar to the Euclidean one (see Section \ref{sec:preliminaries}), a direct analogue of closed hypersurfaces is not available in the Lorentz--Minkowski space, since no closed spacelike hypersurfaces exist (see, e.g., \cite[Proposition~2.5]{BIA2} or \cite{Lopez2014}).

A natural candidate is provided by hypersurfaces that are \emph{asymptotic to a future light cone}. 
These can be described as entire hypersurfaces of the form
\[
\Sigma = \{ (f(x),x)\, ;\ x\in\mathbb{R}^m\}
\]
such that
\begin{align}\label{eq:naive_ALC}
    |f(x) - |x|| \to 0, \qquad \text{as} \ |x|\to\infty.
\end{align}
However, there is a more geometric characterization, based on basic notions from causality theory (see Section~\ref{sec:preliminaries} for precise definitions).

\begin{Definition}
A topological hypersurface $\Sigma\subset\LM^{m+1}$ is said to be \emph{asymptotic to a future light cone} (\emph{ALC}, for short) if its domain of dependence is the chronological future of a point, that is, if $\mathcal{D}(\Sigma) = I^+(o)$ for some $o\in\LM^{m+1}$.
\end{Definition}

Spacelike ALC hypersurfaces share several features with closed hypersurfaces in $\mathbb{R}^{m+1}$; for instance, they have a surjective Gauss map (see Lemma~\ref{lem:surjective_gauss} below). 
By contrast, they are necessarily non-compact, which contributes to the intrinsic difficulty of the problem.

The questions addressed in this work concern the interplay between the mean curvature of an entire spacelike hypersurface $\Sigma$ and the asymptotic condition $\mathcal{D}(\Sigma)=I^+(o)$. 
More precisely, we seek reasonable assumptions on a function $\bar H$ defined on the ambient space that ensure existence and uniqueness of ALC spacelike hypersurfaces whose mean curvature coincides with the restriction of $\bar H$.

From a PDE viewpoint, this problem amounts to finding a function $u\in C^2(\h^m)$ satisfying 
\begin{equation}\label{eq:PDEprob}
\begin{cases}
\div_{h} \left(\frac{D u}{\sqrt{1-|D u|^2}}\right) + \frac{m}{\sqrt{1-|D u|^2}} = m e^u \bar H(e^u q) & \text{on } \h^m,\\
\mathcal{D}(\Sigma_u) = I^+(o).
\end{cases}
\end{equation}
where $\div_{h}$ and the gradient $D$ are taken with respect to the metric $h$ induced by $\eta$ on $\h^m$.

The interest in this problem is not purely geometric, since spacelike hypersurfaces with prescribed mean curvature play an important role in general relativity and in the Born--Infeld theory of electromagnetism. Moreover, many results have been obtained in the case of spacelike vertical (i.e., Cartesian) graphs (see \cite{BS1982, BIM2024, BDP, BIA, BIA3} and the references therein).

By contrast, to the best of our knowledge, the only available result concerning the existence of entire ALC spacelike radial graphs can be found in \cite{Bayard}, and it deals with the prescribed scalar curvature problem. 

Under assumptions analogous to those of Theorem~\ref{thm:treibergs_wei}, the authors of \cite{BIA2} proved the existence and uniqueness of solutions to the Dirichlet problem
\begin{equation}\label{eq:PDEprob2}
\begin{cases}
-\div_{h} \left(\frac{D u}{\sqrt{1-|D u|^2}}\right) 
+ \frac{m}{\sqrt{1-|D u|^2}} 
= m e^u \bar H(e^u q) & \text{in } \Omega,\\
u=0 & \text{on } \partial\Omega,
\end{cases}
\end{equation}
where $\Omega$ is a given smooth bounded domain of $\h^m$. 

Notice that the equation in \eqref{eq:PDEprob2} differs from \eqref{eq:PDEprob} by the sign of the divergence term. 
This discrepancy originates from an inaccuracy in the derivation of the mean curvature equation in \cite[Section 3]{BIA2}. 
Accordingly, the results obtained in \cite{BIA2} concern solutions of \eqref{eq:PDEprob2}.

\subsection{Main results}

To state our main result, we introduce the future-pointing normalized position vector field $T$ on the chronological future $I^+(o)$ of the origin $o$ (see Subsection~\ref{ss:radial}). 
If $u\in C^1(\h^m)$ is spacelike, let $N$ denote the future-pointing unit normal of its radial graph $\Sigma_u$. 
We define the \emph{tilt function} of $\Sigma_u$ with respect to $T$ by
\[
    w_u := - \langle N,T \rangle.
\]
The tilt function measures the hyperbolic angle between $N$ and $T$; in fact, it is the hyperbolic cosine of that angle. A direct computation shows that
\[
w_u=\frac{1}{\sqrt{1-|D u|^2}}.
\]

In what follows, H\"older norms of functions on $I^+(o)$ are computed with respect to the Riemannian metric $\eta_E$ defined by
\[
    \eta_E = \eta + 2 T^\flat\otimes T^\flat.
\]

\begin{Theorem}\label{thm:main}
Fix a point $o\in \LM^{m+1}$ and let $\bar H\in C^{1,\alpha}(I^+(o))$ be a given positive function. Assume that
\begin{enumerate}[label=(H\arabic*), ref=\textit{H\arabic*}]
    \item\label{hp:1} $T\big(\ell_o \bar H(q\ell_o)\big) \ge 0$ for all $q\in\h^m$,
    \item\label{hp:2} $\Vert \bar H\Vert_{C^1(I^+(o))} \le \Lambda$ for some $\Lambda\ge 0$,
    \item\label{hp:3} there exist two radii $0<l\le 1 \le L$ such that
    \[
        \bar H(ql) < \frac{1}{l}, \qquad 
        \bar H(qL) > \frac{1}{L}, 
        \qquad \text{ for all } q\in\h^m\,.
    \]
\end{enumerate}
Then there exists a function $u\in C^{2,\alpha}(\h^m)$ such that its radial graph $\Sigma_u$ is spacelike and satisfies
\begin{align}\label{eq:problema}
\begin{cases}
H_u = \bar H|_{\Sigma_u}, \\
\mathcal{D}(\Sigma_u) = I^+(o).
\end{cases}
\end{align}
Moreover, $w_u\in L^\infty(\h^m)$ and $\Sigma_u$ lies between $\h^m(l)$ and $\h^m(L)$; that is,
\[
\ln l \le u \le \ln L\,.
\]
\end{Theorem}

This result is almost a direct analogue of the existence statement in Theorem~\ref{thm:treibergs_wei}, except for hypothesis~\eqref{hp:1}, which is not needed in the Euclidean case because, unlike $\h^m$, $\s^m$ is compact.

 To prove the result, we solve Dirichlet problems on a family of geodesic balls forming an exhaustion of $\h^m$, and conclude via a diagonal argument. The main difficulty is to control the ellipticity of the operator, which depends on $\|w_u\|_{L^\infty}$, uniformly as the radius tends to infinity. A key ingredient is therefore a gradient estimate uniform in the radius, established in Lemma~\ref{prop:gradient_estmates} and inspired by Bartnik's seminal work \cite{Bartnik88}. We stress that the gradient estimates of \cite{BIA2} cannot be applied in this setting, as they depend on the diameter of the domain.

Treating ALC hypersurfaces as graphs over $\R^m$ is not a viable alternative. Indeed, if $\Sigma$ is the spacelike graph of $f:\R^m\to\R$ (so that $|\nabla_\delta f|<1$, where $\grad$ denotes the Euclidean gradient), its mean curvature is given by
\begin{align*}
    mH = \div_\delta\!\left(\frac{\nabla_\delta f}{\sqrt{1-|\nabla_\delta f|^2}}\right)\,.
\end{align*}
Uniform ellipticity on all of $\R^m$ is then equivalent to the existence of $\theta>0$ such that $|\nabla_\delta f|\leq 1-\theta$ on $\R^m$. However, this condition is incompatible with the ALC behavior \eqref{eq:naive_ALC}, which forces $|\nabla_\delta f|\to 1$ at infinity. By contrast, when $\Sigma$ is realized as the radial graph of $u:\h^m\to\R$, the ALC condition imposes much milder constraints. In particular, it does not force $|Du|$ to approach $1$, and uniform ellipticity is ensured. The trade-off is that the geometric framework becomes somewhat more involved.

\bigskip

A simple class of functions satisfying hypotheses \eqref{hp:1}--\eqref{hp:3} is given by functions that are constant along Lorentzian rays, namely functions of the form
\[
\bar H(\ell q) = \mathcal{H}(q),
\]
for some positive function $\mathcal{H}\in C^{1,\alpha}(\h^m)$ such that
\begin{align}\label{eq:boundedness_H}
    0 < \inf_{\h^m} \mathcal{H} \le \sup_{\h^m} \mathcal{H} < \infty\,.
\end{align}
In this case, the mean curvature of any solution to \eqref{eq:problema} coincides with $\mathcal{H}$.

This naturally raises the question of whether the boundedness assumption \eqref{eq:boundedness_H} is really necessary for the existence of ALC hypersurfaces with prescribed mean curvature. In particular, one may ask whether existence results can still hold when the mean curvature is allowed to decay at infinity, for instance when it is compactly supported or tends to zero.

At present we are not able to give a complete answer to this question. However, we shall return to this issue later and derive a general constraint on the mean curvature of spacelike hypersurfaces.

\bigskip

We now turn to the issue of uniqueness. Within the class of spacelike hypersurfaces bounded between two hyperboloids and with globally bounded tilt function, which includes the solutions constructed above, uniqueness holds. More precisely:

\begin{Theorem}\label{thm:uniqueness}
Fix a point $o\in \LM^{m+1}$ and let $\bar H\in C^{1}(I^+(o))$ be such that 
    \begin{enumerate}[label=(H\arabic*'), ref=\textit{H\arabic*'}]
        \item\label{hp:1'} $T(\ell_o \bar H(q\ell_o)) \geq c >0$ for all $q\in \h^m$,
        \item\label{hp:2'} $\|\bar H\|_{C^{0,\alpha}(I^+(o))} \leq \Lambda$ for some $\Lambda\geq 0$, $\alpha \in (0,1)$,
    \end{enumerate}
    and let $u,v\in C^2(\h^m)$ be bounded spacelike functions such that $w_u,w_v\in L^\infty(\h^m)$. If $u$ and $v$ satisfy \eqref{eq:problema}, then $u = v$.
\end{Theorem}

In general, uniqueness issues for entire spacelike hypersurfaces are quite delicate; see, for instance, the celebrated paper of Treibergs \cite{Treibergs1982}, where entire spacelike constant mean curvature hypersurfaces are classified by their projective boundary values at infinity. What prevents us from obtaining a fully satisfactory uniqueness result is the lack of an \emph{a priori} estimate for the tilt function, which is, in fact, a global gradient estimate for the height function.

When $\bar H$ satisfies the monotonicity condition $E_0(\bar H) \geq 0$, 
uniqueness follows by a comparison argument on compact sets, as in 
\cite{BaySep}. The hypothesis \eqref{hp:1'}, monotonicity along the $T$ 
direction, is instead tailored to the symmetries of the problem: in the 
limiting case $T(\ell_o \bar H(q\ell_o)) \equiv 0$ for every $q \in \h^m$, 
the equation is invariant under Lorentzian dilations centered at $o$, and 
any two bounded spacelike solutions with $w \in L^\infty(\h^m)$ differ by 
such a dilation.

The proof of Theorem~\ref{thm:uniqueness} relies on the combination of several tools, including the Omori--Yau maximum principle, Schauder estimates, and the universal potential estimates of Kuusi and Mingione \cite{mingione}.

\bigskip

We now return to the question raised above concerning the possible decay of the mean curvature at infinity. The next result provides a global constraint in this direction.

Let $\Sigma$ be a spacelike hypersurface. Its future-pointing unit normal induces the \emph{Gauss map}
\[
N:\Sigma \to \h^m.
\]
Let $A$ denote the shape operator of $\Sigma$, and let $\Sigma^+\subseteq \Sigma$ be the set of points where $A$ is positive definite. The following inequality can be interpreted as a Willmore-type inequality for spacelike hypersurfaces.

\begin{Theorem}\label{thm:willmore}
    Let $\Sigma\subset\LM^{m+1}$ be an entire $C^2$ spacelike hypersurface, set $E_0 := (1,0,\dots,0)$ and let $\phi := - \langle N,E_0\rangle$ be the tilt function with respect to $E_0$. If $N:\Sigma^+ \to \h^m$ is surjective, then
    \begin{align*}
        \int_{\Sigma}\abs H^m \phi^{-m-1} \, dv_\Sigma  \geq |\B^m|,
    \end{align*}
    with equality occurring if and only if $\Sigma$ is a hyperboloid.
\end{Theorem}

Notice that the choice of $E_0$ is conventional; the same inequality, with the same constant, holds for any future-pointing unit timelike vector. As shown later (see Lemma \ref{lem:surjective_gauss}), if $\Sigma$ is ALC then the Gauss map $N:\Sigma^+\to \h^m$ is always surjective and thus Theorem \ref{thm:willmore} applies.

Although the inequality above is sharp, the presence of the tilt function, whose behaviour is not known a priori, makes it difficult to exploit it directly to restrict the possible mean curvatures of ALC hypersurfaces. The following theorem complements this picture by providing an integrability obstruction. We note that it is a corollary of a more general statement (see Lemma \ref{lem:interp}).

\begin{Theorem}\label{thm:L^p}
    Suppose $\Sigma\subset\LM^{m+1}$ is an entire $C^2$ spacelike hypersurface such that $N:\Sigma^+\to\h^m$ is surjective. Let $H$ denote its mean curvature. If $H$ is bounded, then $\abs H^p$ is not integrable with respect to $dv_\Sigma$ for any $1 \leq p\leq m$. 
\end{Theorem}

Thus condition \eqref{eq:boundedness_H}, while probably not necessary, cannot be relaxed easily. 
In particular, if $H$ vanishes at infinity, its decay cannot be too fast. 
In any case, Theorem~\ref{thm:L^p} rules out a large class of functions, including those with compact support.

\bigskip

The paper is organized as follows. In Section \ref{sec:preliminaries} we recall the basic facts on causal theory and on the geometry of entire spacelike hypersurfaces and radial graphs in the Lorentz--Minkowski space. Section \ref{sec:Willmore} contains the proof of the Willmore-type inequality and of the associated non-integrability obstruction. In Section \ref{sec:mainest}, we establish the key gradient estimate. Finally, Section \ref{sec:proofs} contains the proofs of the existence and uniqueness theorems.

\subsection*{Notation}
We list some of the notation used throughout the paper.

\begin{itemize}[left=0pt]
\item[$-$] We denote by $\eta$ the Minkowski and by $\bnabla$ the Levi-Civita connection on $\LM^{m+1}$. 
\item[$-$]  We denote by $h$ the metric on $\h^m$ induced by $\eta$, and by $D$ the associated Levi-Civita connection on $\h^m$.
\item[$-$] For a smooth function $\varphi:\h^m \to \R$, in a local chart $\{x^i\}$ with frame $\{\partial_i\}$, we write $\varphi_i = \partial_i \varphi$, $\varphi_{ij} = \partial_i \partial_j \varphi$ and so on. Indices are always raised and lowered with respect to $h$, e.g.\ $\varphi^i = h^{ij} \varphi_j$. 
\item[$-$] Latin indices range from $1$ to $m$, whereas Greek indices range from $0$ to $m$.

\item[$-$] We denote by $\delta$, or simply by $\cdot$, the Euclidean metric on $\R^m$, and by $\grad$ the Euclidean gradient.

\item[$-$] Given a radial graph hypersurface in $\LM^{m+1}$, we denote by $g$ the induced graph metric, and by $\nabla$ its Levi-Civita connection.

\item[$-$] Tensor norms are denoted by $|\cdot|_g$, $|\cdot|_h$, $|\cdot|_\delta$. The subscript is omitted whenever the underlying metric is clear from the context. For instance, we write $|Du|$ instead of $|D u|_h = \sqrt{h^{ij}u_iu_j}$ and $|\nabla u|$ instead of $|\nabla u|_g=\sqrt{g^{ij}u_iu_j}$.

\item[$-$] Given a vector field $X$, we denote by $X^\flat$ the unique $1$-form such that $X^\flat (V) = \langle X, V \rangle$ for every vector field $V$.

\end{itemize}

\section{Preliminaries}\label{sec:preliminaries}

In this section, we fix the notation and, for the sake of completeness, recall some well-known results from causal theory, entire spacelike hypersurfaces, and spacelike radial graphs. In the final part, we introduce the notion of a hypersurface asymptotic to a light cone (ALC, for short) and show that ALC hypersurfaces are both entire and can be represented as radial graphs.

\subsection{Basic causal theory}

Let $\eta$ denote the Minkowski metric on $\LM^{m+1}$ with signature $(-,+,\dots,+)$, that is,
\begin{align*}
    \eta = \langle \cdot,\cdot \rangle = - (dx^0)^2 + (dx^1)^2 + \dots + (dx^m)^2\,.
\end{align*}
Recall that a vector $q \in \LM^{m+1}$ is called \emph{spacelike} if $\langle q,q \rangle > 0$ or $q = 0$. It is called \emph{timelike} if $\langle q,q \rangle < 0$, and \emph{null} (or lightlike) if $\langle q,q \rangle = 0$ with $q \neq 0$. A vector is said to be \emph{causal} if it is either timelike or null.

\begin{Definition}
    A hypersurface $\Sigma\subset\LM^{m+1}$ is \emph{spacelike} if it is a $C^1$ embedded submanifold and the induced metric is Riemannian, or equivalently, if every tangent vector to $\Sigma$ is spacelike.
\end{Definition}
Let $(t, x) \in \R \times \R^m$ denote a point in $\LM^{m+1}$. 
If $\Sigma$ is spacelike then for every $(t_0,x_0)\in\Sigma$ one has
\begin{equation}\label{eq:cone}
|t - t_0| < |x-x_0| \qquad \text{ for all }(t, x) \in \Sigma \setminus\{(t_0, x_0)\}\,.
\end{equation}

We fix a time orientation on $\LM^{m+1}$ by declaring that a causal vector $q=(q^0,q^1,\dots,q^m)$ is \emph{future-pointing} if $q^0>0$ or, equivalently, if $\langle q,E_0 \rangle < 0$, where $E_0 = (1,0,\dots,0)$. If $q^0 <0$, or $\langle q,E_0 \rangle > 0$, we say that $q$ is \emph{past-pointing}. The \emph{chronological future} of a point $o\in\LM^{m+1}$ is 
\begin{align*}
    I^+(o)=\{\,o+q \; ; \;   q \text{ is timelike and future-pointing}\,\}\,.
\end{align*}

A differentiable curve $\gamma: I \subseteq \mathbb{R} \to \LM^{m+1}$ is said to be \emph{causal} and \emph{future-pointing} (respectively, \emph{past-pointing}) if its velocity vector $\dot{\gamma}(t)$ is causal and future-pointing (respectively, past-pointing) for all $t \in I$. A future-pointing curve $\gamma$ is said to be \emph{inextendible in the future} if $\gamma(t)$ does not converge to any point in $\LM^{m+1}$ as $t \to \sup \,  I$. Inextendibility in the past is defined analogously, and
 $\gamma$ is called \emph{inextendible} if it is inextendible both in the future and in the past.

\subsection{Entire spacelike hypersurfaces}
Let $o\in \LM^{m+1}$. For a fixed unit future-pointing tangent vector $X\in T_o\LM^{m+1}$, define its orthogonal complement as 
\[
X^\perp = \{v \in T_o\LM^{m+1} \, ; \, \eta(v,X) = 0\}.
\]
We regard $X^\perp$ as a spacelike hyperplane, passing through the point $o\in \LM^{m+1}$, and orthogonal to $X$. A topological hypersurface $\Sigma\subset\LM^{m+1}$ is called \emph{entire} if there exist a timelike vector $X$ and a real function $f: X^\perp \to \R$ such that $
\Sigma=\{ x + f(x)X \,;\ x\in X^\perp\}\,.$

Up to a Lorentzian isometry we may always assume $X = E_0$ and identify $\LM^{m+1} \cong \R \times X^\perp \cong \R\times \R^m$, thus every entire spacelike hypersurface is of the form
\begin{align*}
    \Sigma =\{ (f(x),x) \,;\ x\in \R^m\}\,.
\end{align*}
When it is useful to emphasize the dependence on $f$, we write $\Sigma_f$.

Introducing the \emph{graph map}
\begin{align}\label{eq:graph_map_cartesian}
    \Gamma_f: \R^m \to \LM^{m+1}, \qquad \Gamma_f(x) = (f(x),x), \qquad \Sigma_f = \Gamma_f(\R^m)\,,
\end{align}
we define the \emph{graph metric} on $\R^m$ by
\begin{align*}
    g_f = \Gamma_f^*\eta = -df^2 + \delta\,,
\end{align*}
so that $(\R^m,g_f)$ is isometric to $\Sigma_f$ endowed with the induced metric. When no confusion arises, we omit the subscript $f$ and simply write $g$ and $\Gamma$. Since $\Sigma_f$ is spacelike if and only if $g$ is Riemannian, it follows that
\begin{align*}
    \Sigma_f\ \text{is spacelike} \ \iff \ |\grad f| < 1\,.
\end{align*}

Let $\{E_0, E_1, \ldots E_m\}$ denote the standard basis on $\R^{m+1}\cong\LM^{m+1}$ with $\langle E_0,E_0\rangle = -1$. If we define the \emph{tilt function} $\phi$ of $\Sigma_f$ with respect to $E_0$ as
\begin{align*}
    \phi := - \langle E_0,N\rangle =\frac{1}{\sqrt{1 - |\grad f|^2}}\,,
\end{align*}
then the future-pointing unit normal of $\Sigma_f$ is given by
\begin{align}\label{eq:normal_cartesian}
    N = \phi(E_0 + \grad f)\,. 
\end{align}

The tilt function is closely related to the density of the volume element of the graph metric. Indeed, since in Cartesian coordinates, the components of the graph metric and its inverse are
\begin{align*}
    g_{ij} = \delta_{ij} - f_i f_j\,, \qquad g^{ij} = \delta^{ij} + \phi^2 f^i f^j,
\end{align*}
by the matrix determinant lemma we have $\det g_{ij} = 1 - |\grad f|^2 = \phi^{-2}$. Thus, denoting by $dv_g$ and $dx$ the volume elements of $g$ and $\delta$, respectively, we get
\begin{align}\label{eq:volumes_cartesian}
dv_g = \phi^{-1} \, dx\,.
\end{align}

We now turn to the extrinsic geometry of $\Sigma_f$. The \emph{second fundamental form} of $\Gamma$ is the bilinear form on $\R^m$
\[
\second(X,Y)
=-\langle \bnabla_{\Gamma_*X}\Gamma_*Y,N\rangle
=\langle \Gamma_*X,\bnabla_{\Gamma_*Y}N\rangle,
\]
for $X,Y\in\mathfrak{X}(\R^m)$, where $\bnabla$ is the Levi-Civita connection of $\eta$. Since the ambient space $\LM^{m+1}$ is flat, $\bnabla$ coincides with the standard derivative in Cartesian coordinates. 
A direct computation using \eqref{eq:normal_cartesian} then gives
\begin{align}\label{eq:second_hessian}
    \second = \phi \, \hess f.
\end{align}

The \emph{shape operator} $A$ is the $(1,1)$ tensor associated to $\second$ via the metric $g$, namely $g (AX,Y) = \second(X,Y)$ for any $X,Y\in\mathfrak{X}(\R^m)$. Equivalently,
\begin{align*}
    AX= \bnabla_{\Gamma_*X}N.
\end{align*}
The mean curvature $H$ and the Gauss--Kronecker curvature $K$ of $\Sigma$ are, respectively, the normalized trace and the determinant of $A$, or equivalently the normalized $g$-trace and the $g$-determinant of $\second$. In coordinates,
\begin{align}\label{eq:mean_and_gauss}
    mH = \tr_g\second=g^{ij}\second_{ij}, \qquad K ={\det}_g\second= \det(g^{ik}\second_{kj}).
\end{align}

\subsection{Spacelike radial graphs}\label{ss:radial}

A topological hypersurface $\Sigma\subseteq \LM^{m+1}$ is said to be a \emph{radial graph} with respect to a point $o \in \LM^{m+1}$ if there exists a function $u: \h^m \to \R$ such that
\begin{align}\label{eq:entire_surface}
    \Sigma = \{ o + qe^{u(q)}\; ; \; q\in \h^m\}\,.
\end{align}
Also in this case, when it is useful to emphasize the dependence on $u$, we write $\Sigma_u$. 

Observe that every radial graph is contained in the chronological future $I^+(o)$. Consider the radial projection onto the hyperboloid $\h^m$
\begin{align*}
    \pi: I^+(o) \longrightarrow \h^m, \qquad \pi(p) = \frac{p-o}{\ell_o(p)}\,,
\end{align*}
where $\ell_o$ denotes the Lorentzian distance from $o$ (see \eqref{eq:time_separation}). When the origin is fixed, we will omit the subscript $o$. If we let $\bar h = \pi^* h$ be the pullback of the hyperbolic metric on $\h^m$, then the ambient metric $\eta$ can be written in polar coordinates as
\begin{align*}
    \eta = - d\ell^2 + \ell^2 \bar h.
\end{align*}
It is convenient to introduce the \emph{time function} $\tau = \ln \ell$, so that 
\begin{equation}\label{eq:metric_eta}
    \eta = e^{2\tau}(-d\tau^2 + \bar h)\,.
\end{equation}

Given a function $u:\h^m\to \R$, consider its radial graph in $I^+(o)$
\begin{equation}\label{eq:Fradial}
    F_u: \h^m\to \LM^{m+1}, \qquad F_u(q) =o + qe^{u(q)}, \qquad \Sigma_u = F_u(\h^m).
\end{equation}
The function $u$ is called the \emph{height function} of the radial graph. Recalling that $\ell(q)=1$ for every $q\in\h^m$, it follows that
\[
\ell(F_u(q)) = e^{u(q)}.
\]
Hence the height function can be interpreted as the restriction of $\tau$ to $\Sigma_u$:
\begin{align}\label{eq:height_function}
    u = F_u^*\tau.
\end{align}
From now on, we shall identify $\tau_{|\Sigma_u}$ with $u$.

The \emph{graph metric} is defined by
\begin{align}\label{eq:metric_g}
    g_u := F^*_u\eta = e^{2u}\left(-du^2 + h \right),
\end{align}
so that $(\h^m, g_u)$ is isometric to $\Sigma_u$ with its induced metric. 
For convenience, we will often identify a radial graph $\Sigma_u$ with either its immersion $F_u$ or its height function $u$, omitting the subscript when no confusion arises.

Recall that $D$ is the Levi-Civita connection of $h$ and $\abs{D\varphi}^2 = h^{ij}\varphi_i\varphi_j$. Since $\Sigma_u$ is spacelike if and only if $g_u$ is Riemannian, it follows from \eqref{eq:metric_g} that
\begin{align*}
    \Sigma_u \ \text{is spacelike} \ \iff \ \abs{Du} < 1.
\end{align*}

Let us set
\begin{align}\label{eq:def_T}
    T := -\bnabla \ell = -e^\tau\bnabla\tau
\end{align}
which is a unit, timelike, future-pointing vector field on $I^+(o)$. If $\Sigma_u$ is spacelike, then it admits a well-defined future-pointing timelike unit normal field $N$. As in the introduction, we define the tilt function of $\Sigma_u$ with respect to $T$ by 
\[
w := - \langle T,N\rangle\,.
\]

Every vector field $X$ in $\LM^{m+1}$ can be decomposed into its tangential and normal components along $\Sigma$ as
\begin{align*}
    X = F_* X^\top - \langle X,N\rangle N\,.
\end{align*}
where, in our convention, $X^\top$ is tangent to $\h^m$.

From \eqref{eq:height_function}, \eqref{eq:def_T}, and the definition of $w$, the splitting of $T$ is
\[
T = -e^uF_*\nabla u + wN, \qquad T^\top = -e^u\nabla u.
\]
As a consequence,
\begin{equation}\label{eq:normal_radial}
    F_*\nabla u =  e^{-u}(wN- T)\,,
\end{equation}
and therefore 
\begin{align}\label{eq:useful_w}
    \abs{\nabla u}^2 = e^{-2u}(w^2 - 1), \qquad \text{ and }\qquad w^2 = {1 + e^{2u}|\nabla u|^2}\,. 
\end{align}
In local coordinates on $\h^m$, the graph metric and its inverse are
\begin{align}\label{eq:metric_components}
    g_{ij} = e^{2u}(h_{ij} - u_iu_j), \qquad  g^{ij} = e^{-2u}\Big(h^{ij} + \frac{u^i u^j}{1 - |D u|^2}\Big)\,\,,
\end{align}
so that
\[
  \abs{\nabla u}^2 = g^ {ij}u_iu_j= e^{-2u}\Big(\frac{|Du|^2}{1 - |Du|^2}\Big), \qquad \text{ and }\qquad w = \frac{1}{\sqrt{1 - |Du|^2}}.
\]
In particular, $w \ge 1$, and it diverges as $|Du| \to 1$, that is, as the hypersurface approaches the lightlike regime and ceases to be spacelike.

As for the second fundamental form, since $\Sigma_u$ is a radial graph over $\h^m$, we may regard it as a tensor on $\h^m$. Namely, for every $X,Y\in \X(\h^m)$ we set
\begin{align*}
    \second(X,Y) = -\langle \bnabla_{F_*X}F_*Y,N\rangle = \langle F_*Y, \bnabla_{F_*X}N \rangle,
\end{align*}
so that the Gauss formula reads
\[
\bnabla_{F_*X}F_*Y = F_*(\nabla_XY) + \second(X,Y)N\,.
\]
As a consequence, recalling\eqref{eq:def_T} and \eqref{eq:height_function}, we obtain
 \begin{align*}
     (F^*\bnabla^2\tau)(X,Y) &= (\bnabla^2\tau)_F(F_*X,F_*Y) = X(du(Y)) - (d\tau)_F(\bnabla_{F_*X}F_*Y) \\
     &=X(du(Y)) - du(\nabla_XY) - \second(X,Y) \langle\bnabla\tau,N\rangle \\
     &=\nabla^2u(X,Y) - we^{-u}\second(X,Y)\,,
 \end{align*}
 that is
\begin{align*}
    we^{-u}\second = \nabla^2 u - F^*\bnabla^2\tau\,.
\end{align*}
Thus, to compute the second fundamental form, it suffices to evaluate the pull-back $F^*\bnabla^2\bar\tau$. For this purpose (and also for later use), it is convenient to work in the coordinates described in the following Lemma. These formulas follow from a straightforward computation, which we omit for brevity. Here and in the following, indices $i,j,k$ range over $1,\dots,m$, while $\mu,\nu$ range over $0,\dots,m$.

\begin{Lemma}\label{lem_local_charts}
    Let $\pi:I^+(o)\to\h^m$ be the radial projection, and fix a global chart $\{x^i\}$ on $\h^m$. Consider the global chart $\{\bar x^\mu\}$, $\mu = 0,1,\dots, m$, on $I^+(o)$ given by
\[
\bar x^0 = \tau\,, \qquad \bar x^i = \pi^* x^i\,, \ i= 1,\dots, m\,.
\]
Then the Minkowski metric $\eta$ in this chart is 
\begin{align*}
    \eta_{\mu\nu} = e^{2\tau} (\bar h_{\mu\nu} - \tau_\mu\tau_\nu)
\end{align*}
and its Christoffel's symbols are
\[
    \begin{aligned}
        \overline{\Gamma}^0_{\mu\nu} &= \bar h_{\mu\nu} + \tau_\mu\tau_\nu\,, \\
        \overline{\Gamma}^k_{\mu\nu} &= \tau_\mu\delta^k_\nu + \tau_\nu\delta^k_\mu + \pi^*(\Gamma^k_{ij})\delta^i_\mu\delta^j_\nu\,,
    \end{aligned}
\]
    where $\pi^*(\Gamma^k_{ij})$ denotes the pullback of the Christoffel symbols $\Gamma^k_{ij}$ associated with the hyperbolic metric $h$ on $\h^m$ in the chart $\{x^i\}$. Moreover, denoting by $\bpart_\mu$ the frame induced by $\{\bar x^\mu\}$, we have
    \[
        \bar h_{\mu\nu} = \pi^*(h_{ij})\delta^i_\mu \delta^j_\nu \qquad \bpart_\lambda \bar h_{\mu\nu} = \pi^*(\partial_k h_{ij})\delta^k_\lambda \delta^i_\mu \delta^j_\nu.
    \]
\end{Lemma}

In the coordinates described above we have
\begin{align*}
    \bnabla^2_{\!\mu\nu\,}\tau = \tau_{\mu\nu} - \bar\Gamma^\lambda_{\mu\nu}\tau_\lambda = - \bar\Gamma^0_{\mu\nu},
\end{align*}
that is 
\begin{align}\label{eq:hessian_tau}
   \bnabla^2\tau= - d\tau^2 - \bar h\,.
\end{align}
As a consequence, by \eqref{eq:metric_g} we have
\[
   F^* \bnabla^2\tau = F^*(- d\tau^2 - \bar h)= -du^2 - h = - 2 du^2 - e^{-2u}g\,,
\]
and the second fundamental form turns out to be
\begin{align*}
    we^{-u}\second = \nabla^2 u + 2du^2 + e^{- 2u}g\,.
\end{align*}
Taking $g$-traces in  previous formula (see \eqref{eq:mean_and_gauss}) immediately gives the intrinsic mean curvature equation
\begin{align}
    \label{eq:intrinsic_mean_curvature}
    mwe^{-u} H = \Delta_g u +2\abs{\nabla u}^2 + me^{-2u}.
\end{align}

A standard computation (use \eqref{eq:metric_components}) shows that the relation between the Hessians of $u$ in the metrics $g$ and $h$ is
\begin{align*}
    \nabla^2 u = w^2D^2 u + w^2\abs{Du}^2 (h - du^2) - 2du^2\,,
\end{align*}
hence
\begin{align*}
    e^{u}\second = we^{2u}(D^2u + h - du^2)\,.
\end{align*}
Taking $g$-traces we finally obtain the mean curvature equation in hyperbolic coordinates
\begin{gather}\label{eq:hyperbolic_mean_curvature}
    m(e^u H - w) = \div_h(wDu) = w\Delta_h u + w^3 D^2u(Du,Du).    
\end{gather}

We conclude this section deriving useful expressions of $N$ in terms of $\nabla u$ and $Du$. Using coordinates as in Lemma \ref{lem_local_charts} the differential of $F$ is given by
\begin{align*}
    F_*\partial_i = \bpart_i + u_i \bpart_0 \, .
\end{align*}
In particular, since $T = e^{-\tau}\bpart_0$, it holds
\[
    F_*\nabla u = \nabla u + e^u|\nabla u|^2 T\,.  
\]
Using also that 
\begin{align}\label{eq:relation_gradients}
    \nabla u = e^{-2u}w^{2}Du\,, 
\end{align}
from \eqref{eq:normal_radial} and \eqref{eq:useful_w} we obtain
\begin{align}\label{eq:normal_radial2}
    N = wT + w^{-1}e^{u} \nabla u= w\left(T + e^{-u}Du\right).
\end{align}

\subsection{ALC hypersurfaces}
Given a subset $A\subseteq\LM^{m+1}$, the \emph{future domain of dependence} of $A$ is the set $\mathcal{D}^+(A)$ consisting of all points $q\in\LM^{m+1}$ such that every inextendible in the past causal curve starting at $q$ intersects $A$. The \emph{past domain of dependence} $\mathcal{D}^-(A)$ is defined analogously. The set
\begin{align*}
    \mathcal{D}(A) = \mathcal{D}^+(A)\cup \mathcal{D}^-(A)
\end{align*}
is called the \emph{domain of dependence of $A$}.

\begin{Definition}
    A topological hypersurface $\Sigma\subset\LM^{m+1}$ is \emph{asymptotic to a future light cone} (\emph{ALC}, for short) if its domain of dependence is the chronologial future of a point, that is, if $\mathcal{D}(\Sigma) = I^+(o)$ for some $o\in\LM^{m+1}$. 
\end{Definition}

Next lemma provides a characterization of ALC hypersurfaces both as entire hypersurfaces and as radial graphs.

\begin{Lemma}\label{L:ALCf}
Let $\Sigma\subseteq \LM^{m+1}$ be a spacelike ALC hypersurface with respect to $o \in \LM^{m+1}$. Then the following properties hold:
\begin{enumerate}[label=$(\roman*)$]
\item $\Sigma$ is an entire spacelike hypersurface. Moreover, if $X\in T_o\LM^{m+1}$ and $f: X^\perp\to \R$ are such that $\Sigma$ is of the form \eqref{eq:entire_surface}, then
\[
\lim_{|x-o|\to \infty} (f(x) - |x-o|) =0
\]
where $x\in X^\perp$ and $|x - o|$ is the intrinsic distance in $X^\perp$.
\item $\Sigma$ is a spacelike radial graph with respect to $o$.
\end{enumerate}
Moreover, if $\Sigma$ is a spacelike radial graph contained in the closed connected region of $\LM^{m+1}$ bounded by two disjoint ALC hypersurfaces with respect to a common point $o \in \LM^{m+1}$, then $\Sigma$ is itself an ALC hypersurface with respect to $o$.
\end{Lemma}
\begin{proof}
We can identify $\LM^{m+1}\cong\R\times \R^m$ in such a way that $X$ and $o$ are mapped to $E_0=(1,0,\dots,0)$ and the zero vector respectively. In particular 
\[
I^+(o) \cong \{(t,x)\in\R^{m+1} \, ; \, t > |x|\}.
\]

Let $\Sigma$ be a spacelike ALC hypersurface. Fix $x \in \R^m$ and consider the vertical line $(t, x)$ for $ t >0$. For $t>0$ big enough this is a causal curve in $I^+(o)$, and by definition of an ALC hypersurface it intersects $\Sigma$ at some point $(f(x),x)$. By \eqref{eq:cone}, such an intersection point is unique. Hence
\[
\Sigma = \Sigma_f = \{(f(x), x) \,;\ x \in \R^m\}\,,
\]
therefore $\Sigma$ is entire. Since $\Sigma$ is spacelike, it follows that $f \in C^1(\R^m)$ and $|\grad f|<1$.

Since $\Sigma \subset I^+(0)$ we have\ $f(x) > |x|$. Fix $\omega \in \mathbb{S}^{m-1}$ and consider
\[
R_\omega(r) = f(r\omega) - r\,, \qquad r >0\,.
\]
Notice that, since $|\grad f|<1$, the function $R_\omega$ is decreasing in $r$ for every $\omega$. We claim that $R_\omega(r) \to 0$ as $r \to \infty$. Indeed, if there exist $C>0$ and $r_0$ such that $R_\omega(r) \ge C$ for all $r \ge r_0$, then the curve $\gamma(r) = (C/2 + r, r\omega)$, $r \geq r_0$, is a future-inextendible causal curve contained in $I^+(0)$ that does not intersect $\Sigma$, contradicting the ALC property.

Next we show that 
\begin{equation}\label{eq:asy}
\lim_{r \to \infty}\sup_{\omega\in\mathbb{S}^{m-1}} R_\omega(r) = 0\,.
\end{equation}
Arguing by contradiction, assume that there exist $C>0$ and sequences $r_j \to \infty$ and $\omega_j \in \mathbb{S}^{m-1}$ such that
\[
R_{\omega_j}(r_j) = f(r_j \omega_j) - r_j > C \quad \text{ for all }j \in \N\,.
\]
Up to a subsequence, we may assume that $\omega_j \to \omega \in \mathbb{S}^{m-1}$. By the previous argument, there exists $\hat r$ such that $R_{\omega}(r) = f(r \omega) - r < C/4$ for all $r \geq \hat r$.

Let $\omega' \in \mathbb{S}^{m-1}$ satisfy $|\omega' - \omega| < \frac{C}{4\hat r}$. Using \eqref{eq:cone}, we obtain
\[
R_{\omega'}(\hat r) = f(\hat r \omega') - \hat r = f(\hat r \omega') - f(\hat r \omega) + f(\hat r \omega) - \hat r < \hat r |\omega' - \omega| + C/4 < C/2\,.
\]
Since $R_\omega$ is decreasing, the above inequality holds for all $r \ge \hat r$. Taking $\omega' = \omega_j$ for $j$ large enough so that $|\omega_j - \omega| < \frac{C}{4\hat r}$ and $r_j \ge \hat r$, we obtain
\[
R_{\omega_j}(r_j)
< C/2\,,
\]
a contradiction. Thus \eqref{eq:asy} holds, which is equivalent to
\[
f(x)-|x|\to 0
\quad \text{as } |x|\to\infty\,.
\] 
This proves $(i)$.

For (ii) fix $q_0 \in \h^m$ and consider the curve $e^u q_0$, $u \in \R$. This is a causal curve, thus by the ALC property there exists $u(q_0)$ such that $e^{u(q_0)} q_0 \in \Sigma$. Uniqueness follows from \eqref{eq:cone}. Since $q_0$ was arbitrary,
\[
\Sigma = \{ qe^{u(q)}  \;;\;  q\in \h^m \},
\]
that is, $\Sigma$ is an entire radial graph. Since $\Sigma$ is spacelike, we have $u \in C^1(\h^m)$ and $|Du|<1$.

Finally, we prove the last statement. Let $\Sigma_1$ and $\Sigma_2$ be two disjoint ALC hypersurfaces. By (ii), $\Sigma_1 = \Sigma_{v_1}$, $\Sigma_2 = \Sigma_{v_2}$ for some functions $v_1, v_2 : \h^m \to \R$ with $v_1 < v_2$. Let now $\Sigma =\Sigma_u$ be a radial graph which lies between $\Sigma_1$ and $\Sigma_2$, that is, such that
\[
v_1 \leq u\leq v_2\,.
\]
Suppose by contradiction that $\Sigma_u$ is not ALC. Then there exists a point $q_0 \in I^+(o)$ and an inextendible causal curve $\Gamma$ through $q_0$ that does not intersect $\Sigma_u$. We can parametrize $\Gamma$ as
\[
\Gamma(s) = o + q(s)e^{\gamma(s)}\,, \quad s \in \R\,,
\]
for suitable functions $q: \R \to \h^m$ and $\gamma: \R \to \R$. Since $\Sigma_1$ and $\Sigma_2$ are ALC, possibly reversing the orientation of $\Gamma$, there exist $s_1 < s_2$ such that
\[
\gamma(s_1) = v_1(q(s_1))\,, \qquad \gamma(s_2) = v_2(q(s_2))\,.
\]
Consider the continuous function
\[
\varphi(s) = u(q(s)) - \gamma(s)\,.
\]
Then $\varphi(s_1) \geq 0$ and $\varphi(s_2) \leq 0$, so there exists $s_* \in [s_1,s_2]$ such that $\varphi(s_*)=0$. This means that $\Gamma$ intersects $\Sigma_u$, providing a contradiction. Hence $\Sigma_u$ is ALC, and the proof is complete.
\end{proof}

\section{Proofs of Theorem \ref{thm:willmore} and Theorem \ref{thm:L^p}}\label{sec:Willmore}

In this section, we prove the Willmore-type inequality stated in Theorem \ref{thm:willmore} and the non-integrability result given in Theorem \ref{thm:L^p}.

In what follows, for a given entire spacelike hypersurface $\Sigma$, we denote by 
\[
\Sigma^+=\{\second\geq 0\}
\]
the set of points where the second fundamental form of $\Sigma$ is positive definite. By a slight abuse of notation, we also use $\Sigma^+$ to denote the subset of $\R^m$ whose points are mapped into $\Sigma^+$ via $\Gamma_f$ (see \eqref{eq:graph_map_cartesian}).

\begin{proof}[Proof of Theorem \ref{thm:willmore}] 
Let $\Sigma = \Sigma_f$ be an entire spacelike hypersurface associated to $f \in C^2(\R^m)$. Using \eqref{eq:normal_cartesian}, we readily check that
\begin{equation}\label{eq:surjective}
         N:\Sigma^+\subseteq\Sigma\to \h^m \text{ is surjective} \quad \iff \quad \grad f: \Sigma^+\subseteq\R^m \to \B^m \text{ is surjective}, 
\end{equation}
where we recall that $\B^m$ denotes the unit ball of $\R^m$.
By the area formula, we then obtain
    \begin{align*}
        |\B^m|  = \int_{\grad f(\Sigma^+)} \, dp \leq \int_{\Sigma^+} \det\hess \, f \, dx. 
    \end{align*}
    Next, observe that since $\det g_{ik} = \phi^{-2}$ (see \eqref{eq:volumes_cartesian}), by \eqref{eq:second_hessian} and \eqref{eq:mean_and_gauss} we have
    \begin{align*}
        {\det}_g\second = \det g^{ik}\second_{kj} = \frac{\det\second_{kl}}{\det g_{ik}} = \phi^2\det(\phi\hess \, f) = \phi^{m+2} \det\hess \, f\,,
    \end{align*}
which gives
    \[
\det \hess f = \phi^{-m-2}{\det}_g \second  = \phi^{-m-2} K\,. 
    \]
Applying the arithmetic-geometric mean inequality
    \begin{align}\label{eq:arithmetic_geometric}
         0 \leq K^{\frac{1}{m}} = \left({\det}_g\second\right)^{\frac{1}{m}} \leq \frac{\tr_g\second}{m} = H\,,
    \end{align}
    which holds at every point of $\Sigma^+$, we get
    \begin{align*}
        |\B^m| \leq \int_{\Sigma^+}\phi^{-m-2}K \, dx \leq \int_{\Sigma^+} \phi^{-m-2} H^m\, dx \leq \int_\Sigma \phi^{-m-1}|H|^m\, dv_g\,,
    \end{align*}
    where in the last step we also used \eqref{eq:volumes_cartesian}. This proves the Willmore-type inequality. 
    
   Now, assume that equality is achieved. Then, by previous computations we infer that $H \equiv 0$ on $\Sigma \setminus \Sigma^+$, while on $\Sigma^+$ we have that $\second = H g$ (by the equality in \eqref{eq:arithmetic_geometric}) and that $\grad f$ is injective. 

    Let us call $P  = \{p \in \Sigma \;;\; H >0\}$. Clearly $P \subset \Sigma^+$, and $P$ is open and non empty (otherwise $H \equiv 0$ on $\Sigma$, contradicting equality). On $\Sigma^+$ we have $\second = H g$, and in particular this holds on $P$. Since Minkowski spacetime is flat, the Gauss-Codazzi equations give
    \begin{align*}
        0 = (\nabla_X\second)Y - (\nabla_Y\second)X = X(H)Y - Y(H)X
    \end{align*}
for all tangent vector fields $X, Y$. Hence, $dH = 0$, i.e., $H$ is constant on every connected component of $P$. Therefore, by continuity of $H$, the set $P$ is also closed. Thus $P = \Sigma$ and $H > 0$ is constant on all of $\Sigma$.  

Consider now the map $\varphi:\Sigma \to \LM^{m+1}$ given by
    \begin{align*}
        \varphi(q) := q - \frac{N(q)}{H}.
    \end{align*}
    We may regard $\varphi$ as a vector field along $\Sigma$.  Computing its covariant derivative along a tangent vector field $X$, we obtain
\begin{align*}
    \bnabla_X \varphi = X - \frac{1}{H}\bnabla_XN = X - \frac{1}{H} H X = 0\,.
\end{align*}
Hence, $\varphi$ is constant along $\Sigma$, i.e., there exists a point $o \in \LM^{m+1}$ such that
\[
    q = o + \frac{N(q)}{H} \quad \text{for all } q \in \Sigma\,.
\]

It follows that
\begin{align*}
    \langle q-o,q-o\rangle = H^{-2}\langle N(q), N(q) \rangle = - H^{-2} \qquad \text{ for all } q\in \Sigma,
\end{align*}
that is, $\Sigma$ is the hyperboloid $o + \h^m(l)$, where $l =H^{-1}$. The proof is complete.
\end{proof}

We now prove Theorem \ref{thm:L^p}. This will follow from a preliminary sharp local inequality. Denote by $r:\Sigma\to [0,\infty)$ the geodesic distance in $\Sigma$ from a fixed point $p\in \Sigma$, and let $B_\rho = \{ r<\rho\}$ be the geodesic ball of radius $\rho>0$ centred at $p$. We then have the following.

\begin{Lemma}\label{L:first_estimate}
    Let $\Sigma$ be a $C^2$ entire spacelike hypersurface and let $B_\rho^+ = B_\rho \cap \Sigma^+$. Then, for any $\rho>0$, it holds that
    \begin{align}\label{eq:first_estimate}
        \norm{H}{L^m(B_\rho)} \ge  |N(B_\rho^+)|^{\frac{1}{m}}\,.
    \end{align}
\end{Lemma}
\begin{proof}

The proof follows the same scheme as that of Theorem \ref{thm:willmore}, with minor adjustments. Let $\Sigma = \Sigma_f$ be an entire spacelike hypersurface associated to $f \in C^2(\R^m)$. It is well known (see, for instance, \cite[Proposition 7.23]{ONeill1983}) that the Gauss-Kronecker curvature $K$ introduced in \eqref{eq:mean_and_gauss} coincides with the Jacobian determinant of the Gauss map $N: \Sigma \to \h^m$. Thus, by the area formula and \eqref{eq:arithmetic_geometric}, we obtain
   \[
  |N(B_\rho^+)| = \int_{N(B_\rho^+)} \, dv_h \leq \int_{B_\rho^+}K \, dv_g \leq  \int_{B_\rho^+}H^m \, dv_g \leq \int_{B_\rho}|H|^m d v_g \,.
   \]
Estimate \eqref{eq:first_estimate} follows immediately.
\end{proof}

We point out that it is not difficult to characterize the equality case in the above lemma, but it is not needed for the purposes of this work.

Next result is a direct consequence of previous lemma, and it is a slightly more general version of Theorem \ref{thm:L^p}.
\begin{Lemma}\label{lem:interp}
 Let $\Sigma$ be a $C^2$ entire spacelike hypersurface. The following statements hold.
\begin{enumerate}
\item[$(i)$] If $p \geq m$, then 
    \begin{align*}
        \norm{H}{L^p(\Sigma)} \geq  \limsup_{\rho\to\infty} \left(\frac{|N(B_\rho^+)|}{|B_\rho|^{1-\frac{m}{p}}}\right)^{\frac{1}{m}}.
    \end{align*}
\item[$(ii)$] if $N:\Sigma^+\to\h^m$ is surjective, then $H \not \in L^m(\Sigma)$. If in addition $H \in L^q(\Sigma)$ for some $q \in (m, \infty]$, then $H \not \in L^p(\Sigma)$ for every $p \in [1, m]$. 
\end{enumerate}
\end{Lemma}
\begin{proof}
    By H\"older inequality and Lemma \ref{L:first_estimate}, for any $p\geq m$ we have
    \begin{align*}
        |N(B_\rho^+)|^{\frac{1}{m}} \leq \norm{H}{L^m(B_\rho)}\leq |B_\rho|^{\frac{1}{m}\left(1-\frac{m}{p}\right)}\norm{H}{L^p(B_\rho)}  
    \end{align*}
    and $(i)$ readily follows. 
    
For $(ii)$, assume that $N:\Sigma^+\to\h^m$ is surjective. Then $|N(B_\rho^+)| \to |\h^m| = \infty$ as $\rho \to \infty$, and we immediately obtain $H \not \in L^m(\Sigma)$ applying $(i)$ with $p = m$. Now, let $1\leq p< m$ and suppose that $H \in L^q(\Sigma)$ for some $q \in (m, \infty]$. Since $H \not \in L^m(\Sigma)$, by the interpolation inequalities
\[
\begin{aligned}
& \|H\|_{L^m(\Sigma)} \leq \|H\|_{L^p(\Sigma)}^{\frac{p(q-m)}{m(q-p)}} \|H\|_{L^q(\Sigma)}^{\frac{q(m-p)}{m(q-p)}}\,, \quad q \in (m, \infty)\,,\\
& \|H\|_{L^m(\Sigma)} \leq \|H\|_{L^p(\Sigma)}^{\frac{p}{m}}\|H\|_{L^\infty(\Sigma)}^{\frac{m-p}{m}}\,,
\end{aligned}
\]
we immediately conclude that $H \notin L^p(\Sigma)$. 
\end{proof}

Finally, both Theorems \ref{thm:willmore} and \ref{thm:L^p} apply to ALC hypersurfaces.

\begin{Lemma}\label{lem:surjective_gauss}
     Let $\Sigma$ be a $C^2$ ALC hypersurface. Then $N:\Sigma^+ \to\h^m$ is surjective.

     Moreover:
     \begin{enumerate}
     \item[$(i)$] it holds
    \begin{align*}
        \int_{\Sigma} |H|^m\phi^{-m-1} \, dv_g  \geq |\B^m|\,,
    \end{align*}
    with equality if and only if $\Sigma$ is a hyperboloid.
    \item[$(ii)$] if $H \in L^\infty(\Sigma)$, then $H \not \in L^p(\Sigma)$ for all $p \in [1, m]$.
    \end{enumerate}
\end{Lemma}
\begin{proof}
By Lemma \ref{L:ALCf}, $\Sigma$ is an entire hypersurface. Thus, up to isometries, we may assume that $\Sigma = \Gamma(\R^m)$ where $\Gamma = \Gamma_f$ as in \eqref{eq:graph_map_cartesian}, for a spacelike function $f$ satisfying $f(x) - |x| = o(1)$ as $|x|\to\infty$. 

By \eqref{eq:surjective}, it is enough to prove that $\grad f$ is surjective. Let $p\in \B^m$ and consider the map $\varphi_p : \R^m \to \R$ defined by $\varphi_p(x) = f(x) - x\cdot p$, where $\cdot$ denotes the Euclidean scalar product in $\R^m$.
    Since $f$ is spacelike and ALC, we have $f(x) > |x|$ everywhere, and since $|p|<1$ this implies that $\varphi_p(x)\to\infty$ as $|x|\to\infty$. Hence $\varphi_p$ attains a minimum at some point $y\in\R^m$, that is,
    \begin{align*}
        \grad f(y) = p, \qquad \hess \, f(y) \geq 0\,.
    \end{align*}
By \eqref{eq:second_hessian}, the latter condition is equivalent to $y$ lying in $\Sigma^+$. Since this holds for arbitrary $p$, we conclude that $\grad f(\Sigma^+) = \B^m$, which proves the surjectivity.  

The remaining claims follow directly from Theorem \ref{thm:willmore} and Theorem \ref{thm:L^p}.
\end{proof}

\section{Main estimates}\label{sec:mainest}

The mean curvature of a spacelike radial graph $F_u:\h^m\to I^+(o)$ is given by 
\[
    m(e^u H - w) = \div_h(wDu) \,,   
\]
where $w = (1 - |D u|^2)^{-1/2}$, see subsection \ref{ss:radial}.
This is a quasilinear elliptic PDE on $(\h^m,h)$, whose ellipticity depends on $|Du|$ and degenerates as $|Du|\to 1$. In this section we show that, under suitable additional assumptions on $H$, such degeneracy cannot occur.

More precisely, given a function $\bar H$ on $I^+(o)$, we consider solutions to the Dirichlet problem
\begin{align}\label{eq:dirichlet}
    \begin{cases}
        m(e^u F_u^*\bar H - w) = \div_h(wDu) &\quad \text{in $B_r$} \\
        u = 0 &\quad \text{on $\partial B_r$}\,,
    \end{cases}
\end{align}
and prove that they satisfy gradient bounds that are \emph{independent of the radius} of the hyperbolic geodesic ball $B_r\subseteq \h^m$.

\subsection{Height estimates} 
Our first result establishes a uniform $L^\infty$ bound on $B_r$ for the solutions to \eqref{eq:dirichlet}. As a consequence, the corresponding portions of the hypersurfaces are trapped between two hyperboloids of radii $l$ and $L$, from below and above, respectively.

\begin{Proposition}\label{prop:height_estimates}
    Assume that $\bar H\in C(I^+(o))$ satisfies \eqref{hp:1} and \eqref{hp:3} for some $0<l\leq 1 \leq L$. Then every $C^2$ solution $u$ to \eqref{eq:dirichlet} satisfies
    \begin{align*}
        \ln l \leq u \leq \ln L \qquad \text{on $\overline{B_r}$.}
    \end{align*}
\end{Proposition}
\begin{proof}
    By assumption, for each $q\in \h^m$, the map $t\mapsto e^t\bar H(qe^t)$ is non-decreasing in $t$, and $L\bar H(qL)>1$. It follows that 
    \begin{align*}
        e^t \bar H(qe^t)-1 >0 \qquad  \text{ for every } t \geq \ln L\,.
    \end{align*}

Let $p\in\overline{B_r}$ be a point where $u$ attains its maximum. If $p\in \partial B_r$ then $\sup u = 0 = \ln 1 \leq \ln L$, as desired. Suppose instead that $p\in B_r$, and assume by contradiction that $u(p) > \ln L$. Since $u$ is $C^2$, we have
    \begin{align*}
        Du(p) = 0 \qquad D^2u(p)\leq 0\,.
    \end{align*}
    Hence,
    \begin{align*}
       m(e^{u(p)} \bar H(pe^{u(p)}) - 1) =  m(e^{u(p)} F^*_u \bar H(p) - w(p)) = \Delta_h u(p) \leq 0\,,
    \end{align*}
    in contradiction with previous inequality. The lower bound can be proved in a similar way. This completes the proof.
\end{proof}

\subsection{Gradient estimates} Here we establish uniform gradient bounds for solutions to \eqref{eq:dirichlet}. The proof is inspired by the work of Bartnik \cite{Bartnik88}. As a preliminary step, we compute the intrinsic Laplacian of the tilt function $w$ of a spacelike radial graph.
\begin{Proposition} 
    Assume $u\in C^3(\h^m)$ is spacelike, and let $\Sigma_u$ be its radial graph. Then the tilt function $w$ satisfies
   \begin{equation}
    \begin{aligned}\label{eq:laplacian_w}
        \Delta_g w - w\abs{\second}^2 &+ mg(T^\top,\nabla H)  \\
        &= w(me^{-2u} + \abs{\nabla u}^2) - mHe^{-u}(w^2 + 1) - 2e^u\second(\nabla u,\nabla u).
    \end{aligned}
    \end{equation}
\end{Proposition}
\begin{proof}
Using \cite[Proposition 2.1]{Bartnik84}, which expresses the variation of the mean curvature of $\Sigma_u$ in terms of $w$ as well as extrinsic geometric data, and recalling that the Lorentz-Minkowski metric is flat, we obtain
    \[
    \begin{aligned}
  \Delta_g w - &w \abs{\second}^2 + mg(T^\top, \nabla H) \\
  &= \frac{1}{2} (\bnabla_N \lie_T\eta)(e_j,e_j)  - (\bnabla_{e_j}\lie_T\eta)(N,e_j) - \frac{m}{2} H\lie_T\eta(N,N) - \lie_T\eta (e_i,e_j)\second(e_i,e_j),
    \end{aligned}
    \]
    where $e_j = F_*\xi_j$ and $\{\xi_j\}$ is any $g$-orthonormal frame on $\h^m$ and the sum over repeated indices is understood.
    
   By definition
    \begin{align*}
        \lie_T\eta(X,Y) = \langle \bnabla_XT,Y\rangle + \langle X,\bnabla_YT\rangle = \bnabla T^\flat(X,Y) + \bnabla T^\flat (Y,X)\,.
    \end{align*}
 Since \eqref{eq:def_T} and \eqref{eq:hessian_tau} imply
    \begin{align}\label{eq:nabla_T}
        \bnabla T^\flat &= - \bnabla(e^\tau  d\tau) = - e^\tau d\tau^2 - e^\tau \bnabla^2\tau = e^\tau \bar h\,,
    \end{align}
using also that $\bar h$ is symmetric we obtain
    \begin{align*}
        \lie_T\eta = 2e^\tau \bar h, \qquad \qquad F_u^*\lie_T\eta = 2e^u h.
    \end{align*}
    Then, recalling \eqref{eq:metric_components} we infer 
    \begin{align*}
        \sum_{i,j=1}^mF_u^*\lie_T\eta(e_i,e_j)\second(e_i,e_j) &=  2e^ug^{ik}g^{jl}h_{ij}\second_{kl} \\
        &=2e^ug^{ik}g^{jl}(e^{-2u}g_{ij}+u_iu_j)\second_{kl} \\ &=2e^u(e^{-2u}g^{kl}\second_{kl} + \second(\nabla u,\nabla u)) \\
        &= 2me^{-u}H + 2e^u\second(\nabla u,\nabla u)\,,
    \end{align*}
    while, by \eqref{eq:normal_radial2}, \eqref{eq:relation_gradients} and the fact that $h(\cdot, T) = 0$, we have
       \[
       \lie_T\eta(N,N) = 2w^2e^{-u} h\left(Du,Du\right) = 2w^2e^{-u}\abs{Du}^2 = 2e^u\abs{\nabla u}^2\,.
       \] 
    Next we compute the covariant derivative of $\lie_T\eta$. To this end, we first compute $\bnabla \bar h$. In the local chart described in Lemma \ref{lem_local_charts}, its components are
    \begin{align*}
        \bnabla_\lambda \bar h_{\mu\nu} &= \bpart_\lambda \bar h_{\mu\nu} - \bar h_{a \nu}\bar \Gamma^a_{\mu\lambda} - \bar h_{\mu b}\bar\Gamma^{b}_{\lambda\nu} \\
        &= \pi^*(\partial_k h_{ij} - h_{aj}\Gamma^a_{ik} - h_{ai}\Gamma^a_{kj})\delta^i_\mu\delta^j_\nu\delta^k_\lambda - 2\bar h_{\mu\nu}\tau_\lambda - \bar h_{\nu\lambda}\tau_\mu - \bar h_{\lambda\mu}\tau_\nu\,.
    \end{align*}
    Since $D$ is the Levi-Civita connection of $h$, it is metric compatible, hence 
\[
0 = D_kh_{ij} = \partial_k h_{ij} - h_{aj}\Gamma^a_{ik} - h_{ai}\Gamma^a_{kj}\,.
\]
Substituting this into the previous expression we get
\[
\bnabla_\lambda \bar h_{\mu\nu} = - 2\bar h_{\mu\nu}\tau_\lambda - \bar h_{\nu\lambda}\tau_\mu - \bar h_{\lambda\mu}\tau_\nu\,, 
\]
or equivalently
\[
\bnabla \bar h = - d \tau \otimes \bar h - d\tau \odot \bar h\,,
\]
where $\odot$ denotes the symmetric tensor product.
    Therefore,
    \begin{align*}
        \bnabla \lie_T\eta = 2\bnabla(e^\tau \bar h) = 2e^\tau (d\tau \otimes \bar h + \bnabla \bar h) = -2e^\tau d\tau \odot \bar h \,.
    \end{align*}
Finally, since by definition of $T$ and $w$ we have
\[
d\tau(N) = e^{-\tau}w\,, \qquad \tau_\mu =\delta_\mu^1 \,,
\]
we readily infer that 
    \[
        \bnabla_N\lie_T\eta(e_i,e_j) = \bnabla_{e_i}\lie_T\eta(N,e_j) = - 2w \bar h_{ij}.
    \]
Thus, 
    \begin{align*}
        \sum_{j=1}^m\frac{1}{2} (\bnabla_N \lie_T\eta)(e_j,e_j)  - (\bnabla_{e_j}\lie_T\eta)(N,e_j) &=  wg^{ij}h_{ij}= w(me^{-2u}+ \abs{\nabla u}^2)\,.
    \end{align*}
    Putting everything together we obtain \eqref{eq:laplacian_w}.
\end{proof}

We can now prove the gradient estimate. 

\begin{Lemma}\label{prop:gradient_estmates}
    Let $0<l\leq 1 \leq L <\infty$ and $\Lambda,R>0$. There exists a constant $C = C(l,L,\Lambda,R,m)$ such that, for any $r>R$ and for any $u\in C^3(\overline {B_r})$ such that
    \begin{enumerate}[label=$(\text{h}_{\arabic*})$, ref=$\text{h}_{\arabic*}$]
        \item \label{hp_1} $\ln l \leq u \leq \ln L$ in $\overline{B_r}$,
        \item\label{hp_2} the mean curvature of $\Sigma_u$ satisfy $\abs H \leq \Lambda$ and $\abs{\nabla H} \leq \Lambda w$ in $B_r$,
        \item\label{hp_3} $u$ is constant on the boundary $\partial B_r$,
    \end{enumerate}
    it holds
    \begin{align*}
        w = \frac{1}{\sqrt{1 - \abs{Du}^2}} \leq C \qquad \text{in $\overline{B_r}$.}
    \end{align*}
\end{Lemma}
\begin{proof}
    In what follows, $C$ denotes a positive constant that depends on $l$, $L$, $\Lambda$, $R$ and $m$; its value may change from line to line. Additional dependencies will be indicated using subscripts.
    
    Let $\lambda\in \R$, and consider the function 
    \[
    \psi_\lambda = we^{\lambda u}\,.
    \]We denote by $x_\lambda\in \overline{B_r}$ a point where $\psi_\lambda$ attains its maximum.

\medskip
    
   \noindent \textit{Step 1: there exist constants $C >0$ and $\mu_1= \mu_1(l,L ,\Lambda, m) \geq 1$, independent of $r>0$, such that, if $|\lambda| > \mu_1$ and $x_\lambda\in B_r$, then $w\leq C$.}

    \smallskip
    
    Assume that $|\lambda| \geq \mu_1 \geq 1$, and $x_\lambda \in B_r$. Since $x_\lambda$ is an internal maximum point of $\psi_\lambda$, we have that at $x_\lambda$, it holds
    \begin{gather}
        \nabla w = -\lambda w\nabla u, \label{eq:max_pt_1} \\ \label{eq:max_pt_2}
        \Delta_g w - \lambda^2 w\abs{\nabla u}^2 + \lambda w\Delta_g u \leq 0\,.
    \end{gather}

    Now we show that, for every $0<\theta<1$, there exists a constant $C_\theta$ such that
    \begin{align}\label{eq:starting_point}
        (1-\theta)\abs\second^2w - \lambda^2w\abs{\nabla u}^2 - \abs\lambda C_\theta w^3 \leq 0 \qquad \text{at $x_\lambda$.}
    \end{align}
    To see this, first observe that, by \eqref{eq:normal_radial} together with the decomposition $T = F_*T^\top + w N$, we have 
    \[
        T^\top = -e^u\nabla u\,.
    \]
    Moreover, for every $\varphi \in C^1(\h^m)$, the following inequality, which is equivalent to the Cauchy-Schwarz inequality by \eqref{eq:metric_components}, holds
     \[
    \abs{\nabla \varphi} \leq we^{-u}\abs{D\varphi}\,.
    \]
    Thus, recalling that $|Du|<1$ and using also \eqref{hp_2}, we obtain 
    \begin{gather*}
        g( T^\top, \nabla H) \leq e^u|\nabla u||\nabla H|\leq \Lambda w^2 \abs{Du} < \Lambda w^2\,, \\
        \abs{\second(\nabla u,\nabla u)} \leq \abs\second \abs{\nabla u}^2 < \abs\second w^2e^{-2u}\,.
    \end{gather*}
     Substituting these estimates in \eqref{eq:laplacian_w}, applying assumption \eqref{hp_1}, and recalling that $w \geq 1$, $|\lambda| \geq 1$, and $l \leq 1$, we get
    \begin{align*}
        \Delta_g w &\geq w\abs\second^2 - \frac{m(2 + 3\Lambda)} {l^2}w^3 - \frac{2}{l}w^2\abs\second \\
        &\geq (1-\theta) w\abs\second^2 - |\lambda|C_\theta w^3\,,
    \end{align*}
    where the last inequality is a consequence of Young's inequality.
    On the other hand, by \eqref{eq:intrinsic_mean_curvature} and performing                     similar estimates we obtain $\abs{\Delta_g u} \leq Cw^2$. Substituting everything in \eqref{eq:max_pt_2} we obtain \eqref{eq:starting_point}.

    Let now $\kappa_i$ be the eigenvalues of $g^{ik}\second_{kj}$ and call $\kappa$ the one with largest absolute value. Observe that by Young's inequality with weight $\frac{1}{m}$ we have
    \begin{equation}
    \begin{aligned}\label{eq:second_fund_form_decomposition}
        |\second|^2 = \kappa^2 + \sum_{j=2}^m \kappa_j^2 &\geq \kappa^2 + \frac{1}{m-1}\Big(\sum_{j=2}^{m}\kappa_j\Big)^2 \\
        &= \kappa^2 + \frac{(mH-\kappa)^2}{m-1} \geq \left(1+ \frac{1}{m}\right)\kappa^2 - m^2H^2.
    \end{aligned}
    \end{equation}
    Our next goal is to show that for every $\epsilon>0$ there exists a constant $C_\epsilon$ such that
    \begin{align}\label{eq:second_point}
    \left(1+\epsilon\right)\kappa^2w \geq \lambda^2w\abs{\nabla u}^2 - Cw^3 \qquad \text{at $x_\lambda$.}
    \end{align}
    Recalling \eqref{eq:nabla_T}, \eqref{eq:max_pt_1} and \eqref{eq:metric_eta} we compute
    \begin{align*}
        |\nabla w|^2 &= -(F_*\nabla w)\langle T,N\rangle \\
        &= - \langle \bnabla_{F_*\nabla w}T,N\rangle - \langle T,\bnabla_{F_*\nabla w}N \rangle \\
        &= -\bnabla T^\flat(F_*\nabla w, N) - \langle F_*T^{\top},\bnabla_{F_*\nabla w}N \rangle \\
        &= -e^u\bar h(F_*\nabla w, N) + e^u\second(\nabla u,\nabla w) \\
        &= e^{-u}\langle F_*\nabla w,N \rangle + e^ud\tau(F_*\nabla w)d\tau(N) + e^u\second(\nabla u,\nabla w) \\
        &= w\langle \nabla w,\nabla 
         u\rangle - \lambda  we^u\second(\nabla u,\nabla u) \\
        &\leq w\abs{\nabla w}\abs{\nabla u} + \abs\lambda w e^u \abs\kappa \abs{\nabla u}^2 \\
        &= (w + e^u |\kappa|)\abs{\nabla w}\abs{\nabla u}\,,
    \end{align*}
    where we used that $\langle F_*\nabla w,N \rangle = 0$ and $\second(v,v)\leq \abs\kappa\abs{v}^2$. Next, using $\abs{\nabla u} = we^{-u} \abs{Du} < w l^{-1}$ and Young's inequality, we obtain, for any $0<\delta<1$,
    \begin{align*}
        \frac{\abs{\nabla w}^2}{w} &< \frac{w\abs{\nabla w}}{l} + \abs\kappa \abs{\nabla w} \\
        &\leq \frac{\delta}{2} \frac{\abs{\nabla w}^2}{w} + \frac{C_{\delta}}{2} w^3 + \frac{1}{2}\frac{\abs{\nabla w}^2}{w} +\frac{\kappa^2}{2} w\,.
    \end{align*}
    Rearranging the terms and setting $\epsilon = \frac{\delta}{1-\delta}$, we obtain \eqref{eq:second_point}.

    Now, fix $\epsilon = \frac{1}{1+4m}$ and choose $\theta = \theta(m)$ such that $(1-\theta)\left(1+\frac{1}{m} \right) = 1 + \frac{1}{2m}$. By \eqref{eq:second_fund_form_decomposition} and \eqref{eq:second_point}, and \eqref{hp_1} we have
    \begin{align*}
        (1-\theta)\abs\second^2 w \geq \Big(1 + \frac{1}{4m} \Big)\lambda^2 w\abs{\nabla u}^2 - C w^3\,,
    \end{align*}
    which substituted in \eqref{eq:starting_point} gives
    \begin{equation}\label{eq:final_point}
        0\geq \frac{|\lambda|}{4m}\abs{\nabla u}^2 - C w^2\,.
    \end{equation}

    Set $\mu_1 \geq 1$ such that
    \[
    \begin{aligned}
      \mu_1 > \max\{1, 4mL^2C\}\,, \quad \text{ and } \quad 
        \frac{\frac{\mu_1}{4mL^2}}{\frac{\mu_1}{4mL^2} - C} \leq 4\,.
    \end{aligned}
    \]
By \eqref{hp_1} we have $\abs{\nabla u}^2 = e^{-2u}(w^2 -1) \geq L^{-2}(w^2-1)$. Thus, from \eqref{eq:final_point} we obtain that for every $|\lambda | \geq \mu_1$ it holds
    \begin{align*}
        w(x_\lambda)^2 \leq 4  \,.
    \end{align*}
    Since $x_\lambda$ is the maximum point of $\psi_\lambda$, we have $we^{\lambda u} \leq w(x_\lambda)e^{\lambda u(x_\lambda)}$ on $\overline{B_r}$, thus
    \begin{align*}
        w \leq 2e^{\lambda(u(x_\lambda) - u)} \leq 2 \max\{L, l^{-1}\}^{2\abs\lambda}\,,
    \end{align*}
  and the proof of this Step is complete. 

\medskip

\noindent    \textit{Step 2:  there exist constants $C >0$ and $\mu_2 = \mu_2(R,\Lambda,L,l) > 0$, independent of $r>0$, such that, if $\abs\lambda \geq \mu_2$ and both $\psi_\lambda$ and $\psi_{-\lambda}$ attains their maxima on $\partial B_r$, then $w \leq C$.}

\smallskip
    
Assume that both $\psi_{\lambda}$ and $\psi_{-\lambda}$ attain their maximum on $\partial B_r$.
    By \eqref{hp_3}, on $\partial B_r$ we have $\psi_\lambda = c\psi_{-\lambda}$ for some constant $c>0$. Thus the maxima of $\psi_\lambda$ and $\psi_{-\lambda}$ on $\partial B_r$ coincide. Let $x_\lambda$ denote such a common maximum point. Suppose that $\nabla u (x_\lambda) \neq 0$. We will show that this leads to a contradiction provided that $\abs\lambda$ is sufficiently large. 
    
    Since, by assumption, $\partial B_r$ is a level set for $u$, the interior $g$-normal to $\partial B_r$ at $x_\lambda$ is given by either
    \begin{align*}
        n = \frac{\nabla u(x_\lambda)}{\abs{\nabla u(x_\lambda)}} \qquad \text{or} \qquad n= -\frac{\nabla u (x_\lambda)}{\abs{\nabla u(x_\lambda)}}.
    \end{align*}
    In both cases, since $x_\lambda$ is a maximum point of $\psi_{\pm\lambda}$ we have $n(\psi_{\pm\lambda})\leq 0$, which implies
    \begin{align}\label{eq:both_cases}
        |\lambda| w \abs{\nabla u}^2 \leq \abs{\langle \nabla u,\nabla w \rangle}.
                                       \end{align}
    
    We now show that there exists a constant $\mu_2>0$, independent of $\lambda$, such that 
    \begin{align}\label{eq:contro_both_cases}
        \abs{\langle \nabla u,\nabla w \rangle} < \mu_2 w\abs{\nabla u}^2\,,
    \end{align}
    thus contradicting \eqref{eq:both_cases}. Starting from $w^2 = e^{2u}|\nabla u|^2 + 1$, we compute 
    \begin{align*}
        \nabla w = \frac{e^{2u}}{w}\left(\nabla^2u(\nabla u) + |\nabla u|^2\nabla u\right)\,,
    \end{align*}
    and thus
    \[
        \langle \nabla u,\nabla w\rangle = \frac{e^{2u}}{w}\left(\nabla^2u(\nabla u,\nabla u) + |\nabla u|^4\right) = \frac{w^2 -1}{w} \left(\langle n, \nabla_n\nabla u\rangle + |\nabla u|^2\right),
    \]
    regardless of the sign in $n=\pm\frac{\nabla u}{|\nabla u|}$. Choose now a local frame $\{e_i\}$ such that $e_1 = n$, $e_a$, $a = 2, \ldots, m$, are tangent to $\partial B_r$, and the frame is $g$-orthonormal at $x_\lambda$. Then
    \begin{align*}
        \langle n, \nabla_n\nabla u \rangle = \Delta_g u - \sum_{a=2}^m \langle e_a,\nabla_{e_a}\nabla u\rangle = \Delta_g u \mp |\nabla u|\sum_{a=2}^m \langle e_a,\nabla_{e_a}n\rangle.
    \end{align*}
   We recognize the last term as a component of the mean curvature vector of $\partial B$ in $I^+(0)$. More precisely, let $f = F_u|_{\partial B_r}$ (see \eqref{eq:Fradial}) and let $A$ denote its second fundamental form, namely
    \begin{align*}
        A(e_a,e_b) &= \langle f_*e_a,\bnabla_{f_*e_b}n\rangle n - \langle f_*e_a,\bnabla_{f_*e_b}N\rangle N \\
        &= \langle e_a,\nabla_{e_b}n\rangle n - \langle f_*e_a,\bnabla_{f_*e_b}N\rangle N\,, \qquad a,b=2,\dots,m.
    \end{align*}
    The mean curvature vector $H_r$ of $f$ is then
    \begin{align*}
        (m-1)H_r &= \sum_{a=2}^m \langle e_a,\nabla_{e_a} n\rangle n - \sum_{a=2}^m \langle f_*e_a,\bnabla_{f_*e_a}N\rangle N \\
        &= (m-1)\langle H_r,n\rangle n - (m-1)\langle H_r,N\rangle N,
    \end{align*}
    and, in particular, 
    \begin{align*}
        \sum_{a=2}^m \langle e_a,\nabla_{e_a} n\rangle =    (m-1)\langle H_r,n\rangle.
    \end{align*}
    Thus, using \eqref{eq:intrinsic_mean_curvature} we obtain
    \begin{align}\label{eq:scalar_prod}
        \langle\nabla u,\nabla w\rangle = \frac{w^2 - 1}{w}\left(mwe^{-u}H - |\nabla u|^2 - me^{-2u} \mp(m-1) |\nabla u|\langle H_r, n \rangle \right).
    \end{align}
    Since $u$ is constant on $\partial B_r$, $f(\partial B_r)$ is a spacelike round sphere of codimension $2$ in Minkowski spacetime. Hence $H_r$ is a spacelike vector and satisfies $\abs{H_r}\leq \frac{c}{r}$. In particular, if $r\geq R$ then $\abs{H_r}\leq C_R$. 
    Also, by \eqref{eq:normal_radial2} we have
    \begin{align*}
        \langle H_r,n\rangle ^2 = \abs{H_r}^2 + \langle H_r,N\rangle^2 \leq C_R^2(1+e^{2u}w^{-2}\abs{\nabla u}^2).
    \end{align*}
    Substituting this in \eqref{eq:scalar_prod} and using \eqref{hp_1} and \eqref{hp_2} we obtain \eqref{eq:contro_both_cases}.

By \eqref{eq:contro_both_cases} and \eqref{eq:both_cases}, we conclude that $\nabla u(x_\lambda) = 0$ if $|\lambda| \geq \mu_2$, that is, $w(x_\lambda) = 1$. Arguing as in the final part of the proof of Step 1, we obtain that 
    \begin{align*}
        w \leq e^{\lambda(u(x_\lambda) - u)} \leq  \max\{L, l^{-1}\}^{2\abs\lambda}\,,
    \end{align*}
which completes the proof of this Step. 

\medskip

 \noindent   \textit{Step 3: conclusion.}

\smallskip    
    
    Let $\lambda \in \R$ be such that $|\lambda| \geq \max\{\mu_1, \mu_2\}$. If a maximum point of either $\psi_\lambda$ or $\psi_{-\lambda}$ lies in the interior of $B_r$, then $w \leq C$ by Step 1. If, on the other hand, both $\psi_\lambda$ and $\psi_{-\lambda}$ attain their maxima on $\partial B_r$, then $w \leq C$ by Step 2. This completes the proof.
\end{proof}

\section{Proofs of Theorem \ref{thm:main} and Theorem \ref{thm:uniqueness}}\label{sec:proofs}
In this section, we prove the existence result of Theorem \ref{thm:main} and the uniqueness result stated in Theorem \ref{thm:uniqueness}. As a preliminary step, we establish a priori $C^{k,\alpha}$ local regularity estimates for solutions to
\begin{align}\label{eq:hyperbolic_mean_curvature_barH}
    \div_h(wDu) = m(e^u F^*_u\bar H - w)\,,
\end{align}
and, more generally, for spacelike radial graphs with bounded mean curvature. To derive these estimates, it is convenient to rewrite \eqref{eq:hyperbolic_mean_curvature_barH} as a PDE on the unit Euclidean ball $\B^m\subset \R^m$ via the Poincar\'e model of hyperbolic space. More precisely, we identify the hyperboloid $\h^m\cong\h^m(1)$ with $\B^m$ using the hyperbolic stereographic projection $\Phi: \h^m\to \B^m$, which maps each point $q\in\h^m(1)$ to the unique intersection of the line $\{sq \;;\; s\in\R\}$ with the plane $\{q^0 = 1\}$. Explicitly,
\begin{align}\label{eq:stereographic_projection}
    \Phi(q) =  \frac{1}{1+q^0}(q^1,\dots, q^m)\,, \qquad \Phi^{-1}(x) = (\lambda-1, \lambda x)\,,
\end{align}
where
\begin{align*}
    \lambda = \frac{2}{1 - |x|_\delta^2}, \qquad \qquad |x|_\delta^2 = \delta (x,x), \quad x\in \B^m\,,
\end{align*}
and $\delta$ denotes the standard Euclidean metric on $\B^m$. 

The pullback of the hyperbolic metric under $\Phi^{-1}$ is then 
\[
(\Phi^{-1})^*h = \lambda^2\delta\,.
\]
Notice that, by standard invariance of H\"older spaces under smooth diffeomorphisms, for every integer $k \geq 0$, $\alpha \in (0,1)$, and $r <1$, the norms $\|u\|_{C^{k,\alpha}(\Phi^{-1}(B_r^\delta))}$ and $\|u \circ \Phi^{-1}\|_{C^{k,\alpha}(B_r^\delta)}$ are equivalent, with constants depending only on $m$, $k$, $\alpha$, and $r$.

In what follows we will denote the Euclidean metric $\delta$ by a dot $\cdot$ and, with a slight abuse of notation, we will identify $h=(\Phi^{-1})^*h$ and, for any $f\in C(\h^m)$, we will write $f=(\Phi^{-1})^*f$ when considering $f$ as a function on $\B^m$.
\begin{Lemma}
Let $u$ be a spacelike radial graph on $\h^m$ with mean curvature $H$. Then the pullback of $u$ via $\Phi^{-1}$ satisfies
      \begin{equation}\label{eq:mean_conf}
        \div_\delta(w \, \grad u) = m\lambda^2(e^u H - w) - (m - 2) \lambda w \, \grad u\cdot x\, \qquad \text{ in }\B^m.
    \end{equation}

Moreover, \eqref{eq:mean_conf} can be written in non-divergence form as
    \begin{equation}\label{eq:nondiv}
         Pu = f_H(x,u, \grad u) \qquad \text{ in } \B^m\,,
    \end{equation}
    where 
    \begin{equation}\label{eq:Pop}
    Pu = a_{kl}(x,\grad u)u_{kl} + b_k(x,\grad u) u_k
    \end{equation}
with coefficients given, for all $x\in\B^m$, $y \in \R$, and $z\in\mathbb{R}^m$, by
    \begin{align*}
        &a_{kl}(x,z) = (1 - \lambda^{-2}|z|^2) \delta_{kl} + \lambda^{-2} z_k z_l, \\ 
        &b_k(x,z) = \lambda \big[(m-2) - (m-1)\lambda^{-2}|z|^2\big] x_k, \\ 
        &f_H(x,y, z) = m\lambda^2 \big[(1 - \lambda^{-2}|z|^2)^{3/2}e^y H - (1 - \lambda^{-2}|z|^2)\big]\,.
    \end{align*}
 
\end{Lemma}

\begin{proof}
Recall the conformal transformation identities
    \begin{align*}
        \div_h X = \lambda^{-m} \div_\delta (\lambda^mX)\,, \qquad Du = \lambda^{-2}\grad u\,.
    \end{align*}
  Since $\grad \lambda = \lambda^2 x$, we obtain
\[
\lambda^2 \div_h(wDu) = (m-2)  \lambda w\, \grad u \cdot x + \div_\delta(w\grad u)\,,
\]
and \eqref{eq:mean_conf} follows from \eqref{eq:hyperbolic_mean_curvature}.

Next, observing that $|Du|^2 = \lambda^{-2}|\grad u|^2$, we compute
\begin{align*}
        \grad w = w^3 \lambda^{-2} (\grad^2u(\grad u,\cdot) - \lambda|\grad u|^2 x )\,.
\end{align*}
Thus, also using the identity $w^2\lambda^{-2}|\grad u|^2 = w^2-1$, we find
    \begin{align*}
 \div_\delta(w \grad u) = w \Delta_\delta u + w^3 \lambda^{-2}\grad^2u(\grad u, \grad u) - (w^3-w)\lambda \grad u \cdot x\,.
    \end{align*}
The non-divergence form \eqref{eq:nondiv} follows by substituting into \eqref{eq:mean_conf} and dividing by $w^3$.
\end{proof}

We now prove the main regularity lemma. The proof relies on \cite[Theorem 1.4]{mingione}, and adapts the strategy of \cite[Theorem 1.6]{BIA} 

\begin{Lemma}\label{lemma:mingione}

Let $r > 0$, and let $u \in C^{2}(\Phi^{-1}(B^\delta_r))$ be a spacelike radial graph over $\h^m$ with mean curvature $H$. Assume that in $\Phi^{-1}(B^\delta_r)$
\[
\ln l \leq u \leq \ln L, \qquad |Du| \leq 1 - \theta\,,
\]
for some constants $l \leq 1 \leq L$ and $\theta \in (0,1)$. Then the following hold:
\begin{enumerate}[label=$(\roman*)$]
\item If $\|H\|_{L^\infty(\Phi^{-1}(B^\delta_r))} \leq \Lambda$, then for every $\alpha \in (0,1)$ there exists a constant $C_r= C_r(m,r,\alpha,\Lambda,l, L, \theta)$ such that
\[
\|u\|_{C^{1,\alpha}(\Phi^{-1}(B^\delta_{r/2}))} \leq C_r.
\]
\item If $\|H\|_{C^{k,\alpha}(\Phi^{-1}(B^\delta_r))} \leq \Lambda$ for some integer $k \geq 0$ and $\alpha \in (0,1)$, then 
\[
u \in C^{k+2,\alpha}(\Phi^{-1}(B^\delta_{r/2}))\,,
\]
and there exists a constant $C'_r = C'_r(m,r,k, \alpha,\Lambda,l, L, \theta)$ such that 
\[
\|u\|_{C^{k+2,\alpha}(\Phi^{-1}(B^\delta_{r/2}))} \leq C'_r.
\]
\item If the assumptions in $(i)$ or $(ii)$ hold for every $r > 0$, with $\Lambda$, $l$, $L$, and $\theta$ independent of $r$, then the constants $C_r$ and $C'_r$ can be chosen independently of $r$.
\end{enumerate}
\end{Lemma}
\begin{proof}
Since $\|u\|_{L^\infty(B^\delta_r)} \leq \max\{|\ln l|, \ln L\}$ by assumption, to prove $(i)$  it suffices to show that for all $\alpha \in (0,1)$, there exists a constant $C = C(m,r,\alpha,\Lambda, L, \theta)$ such that 
    \begin{equation}\label{eq:c1alpha}
    \sup_{\substack{x,y \in B_{r/2}^\delta \\x \neq y}} \frac{|\grad u(x) - \grad u(y)|}{|x-y|^\alpha} \leq C.
    \end{equation}

Define 
\[
a(x, z)  = \frac{1}{\sqrt{1-\varphi(x, |z|)^2}}\,,
\]
where $\varphi: B_r^\delta \times [0, \infty)\to \R$ is a smooth function such that 
\[
\varphi(x, s) = 
\begin{cases}
\lambda^{-1}(x)s & \qquad \text{ if } \lambda^{-1}(x)s < 1 - \frac{\theta}{2}\\
1 - \frac{\theta}{4} & \qquad \text{ if } \lambda^{-1}(x)s > 1 - \frac{\theta}{4}
\end{cases}
\]
and for each $x \in B_r^\delta$ the map $s \mapsto \varphi(x, s)$ is non-decreasing. Since $u$ solves \eqref{eq:mean_conf} and satisfies $|Du| = \lambda^{-1}|\grad u| \leq 1 - \theta$, it follows that 
\[
-\div_\delta A(x, \grad u) = \mu \quad \text{in } B_r\,,
\]
where 
\[
\begin{aligned}
&A(x, z) = a(x, z)z\,,\\
&\mu = -m\lambda^2 (e^u H - w) + (m - 2)\lambda w \, \grad u \cdot x\,.
\end{aligned}
\]
Arguing exactly as in \cite[Proof of Theorem 1.6]{BIA}, we check that $A$ satisfies the growth condition \cite[(1.2)]{mingione}. Next, observe that if $|z| > \tilde r :=\frac{2}{1-r^2}(1 - \frac{\theta}{4})$, then $x \mapsto a(x, z)$ is constant in $B_r^\delta$. As a consequence, for every $x, y \in B^\delta_r$ and $z \in \R^m$, we have
\[
|a(x, z) - a(y, z)| \leq C_{r,\theta}|\varphi(x, |z|)- \varphi(y, |z|)| \leq C_{r, \theta}\|\varphi\|_{C^{0,\alpha}(B_r^\delta \times [0, \tilde r])}|x-y|\,.
\]
In particular, $a$ is Lipschitz continuous in $x$, and therefore $A$ satisfies the Dini-H\"older continuity assumption required in \cite[Theorem 1.4]{mingione}, for any $\alpha \in (0,1)$. Finally, since $e^u\leq L$ and $\norm{H}{L^\infty(B^\delta_r)}\leq \Lambda$, it follows that $|\mu| \leq c$ for a constant depending only on $\Lambda, L,\theta,r$. Therefore, there exist positive constants depending only on $\Lambda, L,\theta,r$ and $\alpha$ such that
    \begin{align*}
        \fint_{B^\delta_r}|\grad u| \, dx \leq c, \qquad \text{ and }\qquad \int_{0}^r \frac{1}{\rho^{m + \alpha}}\int_{B^\delta_\rho} |\mu| \, dx \, d\rho \leq c\int_0^r \frac{d\rho}{\rho^{\alpha}} < \infty\,.
    \end{align*}
Then \eqref{eq:c1alpha} follows from \cite[Theorem 1.4]{mingione} and $(i)$ is proved.

We turn to the proof of $(ii)$. Since $\lambda^{-1}|\grad u| \leq 1-\theta$, the operator $P$ in \eqref{eq:Pop} is uniformly elliptic on $B_r^{\delta}$, with ellipticity constant depending only on $\theta$. By classical Schauder estimates, for every $k \ge 0$ and $\alpha \in (0,1)$, it holds 
\[
\|u\|_{C^{k+2,\alpha}(B^\delta_{r/2})}
\le C\Big(
\|f_H(x,u,\grad u)\|_{C^{k,\alpha}(B^\delta_r)} 
+ \|u\|_{L^\infty(B^\delta_r)}
\Big),
\]
where $C>0$ is a constant depending only on $m$, $r$, $\theta$, $\alpha$ and the norms $\|a_{ij}(x,\grad u)\|_{C^{k,\alpha}(B^\delta_r)}$ and $\|b_{j}(x,\grad u)\|_{C^{k,\alpha}(B^\delta_r)}$.
Let $k=0$ and fix $\alpha \in (0,1)$. Thanks to $(i)$, the quantities $\|a_{ij}(x,\grad u)\|_{C^{0,\alpha}(B^\delta_r)}$, $\|b_{ij}(x,\grad u)\|_{C^{0,\alpha}(B^\delta_r)}$ and $\|f_H(x,u,\grad u)\|_{C^{0,\alpha}(B^\delta_r)}$ are uniformly bounded, and hence $u \in C^{2,\alpha}(\Phi^{-1}(B^\delta_{r/2}))$. The conclusion for general $k \geq 1$ follows by a standard bootstrap argument. 

Suppose now that the assumptions on $H, u, Du$ hold on the whole $\h^m$ with uniform constants $\Lambda, L, \theta$. Fix any point $q\in\h^m(1)$, and let $\Psi$ be a Lorentzian boost mapping $q$ to $(1,0,\dots,0)\in\LM^{m+1}$. Let $\psi:\B^m\to\B^m$ be the map making the diagram
    \[
    \begin{tikzcd}
        \h^m \arrow{r}{\Psi} \arrow[swap]{d}{\Phi} & \h^m \arrow{d}{\Phi} \\
        \B^m \arrow{r}{\psi}& \B^m
    \end{tikzcd}
    \]
    commute with the stereographic projection $\Phi$ (see \eqref{eq:stereographic_projection}). By construction, $\psi$ is an isometry of $(\B^m,\lambda^2\delta)$ that maps $\Phi(x)\cong x$ to the origin. The pullback $\psi^*u$ solves $P(\psi^*u) = f_{\psi^*H}$, and the bounds on $u$, $\grad u$, and $H$ are preserved. Under the assumptions of $(i)$, for instance, we obtain
    \begin{align*}
        \norm{u}{C^{1,\alpha}(\psi^{-1}(B^\delta_{1/4}))}&=\norm{\psi^*u}{C^{1,\alpha}(B^\delta_{1/4})} \leq C_{1/4}(m, \alpha, \Lambda, l, L, \theta)\,,
    \end{align*}
    with a constant independent of $r$. Since every point $q \in \h^m$ can be mapped to the origin and the sets $\psi^{-1}(B^\delta_{1/4})$ cover $\h^m$ with uniform control, the global estimate follows. This concludes the proof.
\end{proof}

\begin{proof}[Proof of Theorem \ref{thm:main}]  Let $\bar H \in C^{1,\alpha}(I^+(o))$ satisfy \eqref{hp:1}-\eqref{hp:3}. We first verify that $\bar H$ satisfies the mean curvature structure conditions (MCSC) introduced in \cite{Bartnik88}. Let $\Sigma$ be a spacelike radial graph with future-directed unit normal $N$ and tilt function $w$. By \eqref{hp:2} we have that $ |\bar H| \leq \Lambda$ and $|d\bar H|_{E} \leq \frac{\Lambda}{2}$ on $I^+(o)$, where the norm is computed with respect to the Riemannian metric $\eta_E = \eta + 2(T^\flat)^2$. Observe that 
    \begin{align*}
        \abs{N}_E^2 = \langle N,N\rangle + 2\langle T,N\rangle ^2 = 2w^2 - 1 \leq 2w^2\,.
    \end{align*}
Using also that $|T|_E = 1$, we estimate the gradient of the restriction $H = \bar H_{|\Sigma}$ as
    \begin{align*}
        \abs{\nabla H}^2 &= \langle \bnabla \bar H,\bnabla \bar H \rangle + d\bar H(N)^2 \\
        &= |\bnabla \bar H |_E^2 - 2 d\bar H(T)^2 + d\bar H(N)^2 \leq \Lambda^2 w^2\,.
    \end{align*}
In particular, $|H| \le \Lambda$ and $|\nabla H| \le \Lambda w$, so that $\bar H$ satisfies (MCSC).

For each integer $j \ge 1$, let $B_j$ be the geodesic ball in $\h^m$ of radius $j$ centred at a fixed point $x_0 \in \h^m$, and consider the Dirichlet problem
\begin{align}\label{eq:dirichlet_j}
    \begin{cases}
        \div_h(wDu) = m(e^u F_u^*\bar H -w) &\quad \text{in $B_j$} \\
        u = 0 &\quad \text{on $\partial B_j$}.
    \end{cases}
\end{align}
Since $\bar H$ satisfies (MCSC), by \cite[Theorem 4.1]{Bartnik88}, we get that there exists a spacelike solution $u_j \in C^{3,\alpha}(B_j)$ to \eqref{eq:dirichlet_j}.

Next, using \eqref{hp:1}-\eqref{hp:3} and Proposition \ref{prop:height_estimates}, we infer that $\ln l \leq u_j \leq \ln L$. Let us set $H_j := F^*_{u_j}\bar H$. Then, as seen above, we check that $|H_j| \le \Lambda$ and $|\nabla H_j| \le \Lambda w_j$, where $w_j$ is the tilt function associated with $u_j$. Since $u_j = 0$ on $\partial B_j$, Lemma \ref{prop:gradient_estmates} gives a uniform bound $w_j \le C$ on $\overline{B_j}$, with $C$ independent of $j$. In particular, this implies the uniform estimate $|Du_j| \le 1 - \theta$ for some $\theta \in (0,1)$ independent of $j$. 

Fix now $R > 0$, let $B_R \subset \h^m$ be a geodesic ball, and take $r > 0$ such that $B_R = \Phi^{-1}(B_r^\delta)$. Applying Lemma \ref{lemma:mingione}-$(i)$ to $u_j$ in $B_r^\delta$ (with $H = H_j$), and then Lemma \ref{lemma:mingione}-$(ii)$ twice, we obtain higher-order regularity estimates. Namely, there exists a positive constant $C = C(m, r, \alpha, \Lambda, l, L, \theta)$ such that
\begin{equation}\label{eq:c3est}
\|u_j\|_{C^{3,\alpha}}(B^\delta_r) \leq C\,.
\end{equation}
By standard precompactness results for H\"older spaces (see \cite[Lemma 6.36]{GilTru}), up to a subsequence we have that $u_j \to u$ in $C^{3,\beta}(B_r^\delta)$ for every $\beta < \alpha$. In fact, thanks to \eqref{eq:c3est}, it holds that $u \in C^{3,\alpha}(B_r^\delta)$. Passing to the limit in the equation $Pu_j =  f_{H_j}(x,u_j, \grad u_j)$, we conclude that $u$ solves
$Pu = f_{F^*_u \bar H}(x, u,  \grad u)$ in $B^\delta_{r}$.
Since $R$ and thus $r$ were arbitrary, a standard diagonal argument gives a global solution $u \in C^{3,\alpha}(\h^m)$ to \eqref{eq:problema}.

Finally, by construction $u$ satisfies $\ln l \le u \le \ln L$ and $|Du|\leq 1-\theta$. In particular, $u$ is a spacelike radial graph and lies between two ALC hypersurfaces. Hence it is itself ALC by Lemma \ref{L:ALCf}. This completes the proof.
\end{proof}

\begin{proof}[Proof of Theorem \ref{thm:uniqueness}]
Let $\bar H\in C^1(I^+(o))$ satisfy \eqref{hp:1'}-\eqref{hp:2'}. Notice that \eqref{hp:1'} is equivalent to 
    \begin{align}\label{hp:stronger}
        \partial_t\left(e^t \bar H(xe^t) \right)\geq c >0 \qquad \text{ for all } x\in\h^m\,.
\end{align}
Let $u,v\in C^{2}(\h^m)$ be bounded solutions to \eqref{eq:problema} such that $w_u,w_v\in L^\infty(\h^m)$, where $w_u=(1 - |Du|^2)^{-1/2}$ and $w_v=(1 - |Dv|^2)^{-1/2}$.
Let us immediately point out that, since $w_u, w_v$ are bounded in $\h^m$, then $|Du| < 1 - \theta_u$, $|Dv| < 1-\theta_v$ in $\h^m$, for some constants $\theta_u, \theta_v \in (0,1)$. Therefore, using also \eqref{hp:2'}, by Lemma \ref{lemma:mingione}-$(iii)$ we infer that 
\[
\|u\|_{C^{2,\alpha}(\h^m)} \leq C \quad \text{ and } \quad  \|v\|_{C^{2,\alpha}(\h^m)} \leq C\,.
\]

Assume by contradiction that $u \neq v$. Without loss of generality, suppose that $\{u >v\} \neq \emptyset$, and set $\varphi = u-v$. Then $\varphi$ is bounded and $\sup \varphi >0$. Applying the Omori--Yau maximum principle to $\varphi$, we obtain a sequence $\{x_j\}\subseteq \h^m$ such that 
    \begin{align*}
        \varphi(x_j) \geq \sup \varphi - \frac{1}{j}\,, \qquad |D\varphi(x_j)|\leq \frac{1}{j}\,, \qquad D^2\varphi(x_j) \leq \frac{h_{x_j}}{j}\,.
    \end{align*}
    Set
    \[
    \Theta(t,x) := e^t \bar H(xe^t)\,.
    \]
    Notice that, thanks to \eqref{hp:stronger}, it holds $\partial_t\Theta(t,x) \geq c >0$. Since both $u$ and $v$ solve \eqref{eq:problema} (see also \eqref{eq:hyperbolic_mean_curvature}), we have
    \begin{equation} \label{eq:e^uH}
     \begin{aligned}
        m(\Theta( u(x_j),x_j) - \Theta(v(x_j),x_j)) &= m(w_u(x_j) - w_v(x_j))  \\ 
        &+ w_u(x_j)\Delta_h u(x_j) - w_v(x_j)\Delta_h v(x_j)  \\ 
        &+w_u^3(x_j)D^2u(Du,Du)(x_j) - w^3_v(x_j)D^2v(Dv,Dv)(x_j).
        \end{aligned}
    \end{equation}
  We next prove that
    \begin{align}\label{eq:limit_of_e^uH}
        \limsup_{j\to\infty}(\Theta( u(x_j),x_j) - \Theta(v(x_j),x_j)) \leq  0\,.
    \end{align}
    First, since $|D\varphi(x_j)| = |Du(x_j) - Dv(x_j)|\to 0$, we have 
    \[
    |w_u(x_j) - w_v(x_j)|\to 0 \quad \text{ as } j \to \infty\,,
    \]
    and the first term in the right-hand side of \eqref{eq:e^uH} vanishes as $j\to \infty$. For the second term, omitting the dependence on $x_j$, we write
    \begin{align*}
        w_u\Delta_h u - w_v\Delta_h v &= (w_u - w_v)\Delta_h u + w_v\Delta_h\varphi \leq |w_u - w_v||\Delta_h u| + w_v \frac{m}{j}\,.
    \end{align*}
    Since $|\Delta_h u| \leq \sqrt{m}|D^2u|$ is bounded, this term also vanishes as $j \to \infty$. As for the third term, we compute
    \begin{align*}
        w_u^3D^2u(Du,Du) - w^3_vD^2v(Dv,Dv)& \\
        = (w_u^3 - w_v^3)D^2u(Du,Du) + w_v^3 \big(&D^2\varphi(Du,Du) + D^2v(D\varphi,Du) + D^2v(Dv,D\varphi) \big) \\
        &\leq c|w_u(x_j) - w_v(x_j)||D^2u| + \frac{w_v}{j} \left( m + 2|D^2v| \right),
    \end{align*}
    which vanishes as $j\to \infty$ because $w_v$ is bounded. Thus \eqref{eq:limit_of_e^uH} is proved.

    By Lagrange's theorem, for every $j$ there exists $s(x_j) \in (v(x_j), u(x_j))$ such that
    \begin{align*}
        \Theta( u(x_j),x_j) - \Theta(v(x_j),x_j) = \partial_t\Theta(s(x_j))( u(x_j) - v(x_j))\,.
    \end{align*}
    Noticing that $\varphi(x_j)> 0$ for $j$ sufficiently large, by \eqref{hp:stronger} and \eqref{eq:limit_of_e^uH} we infer
    \begin{align*}
        0 < \varphi(x_j) = u(x_j) - v(x_j) &= \frac{\Theta( u(x_j),x_j) - \Theta(v(x_j),x_j)}{\partial_t\Theta(s(x_j))}\\
        &\leq \frac{\Theta( u(x_j),x_j) - \Theta(v(x_j),x_j)}{c} \to 0\,,
    \end{align*}
which is a contradiction. Therefore $u = v$, and the theorem is proved.
\end{proof}

\section*{Acknowledgment}
\thanks{We would like to thank Prof. D. Bonehure and Prof. A. Seppi for the useful discussions.}

\end{document}